\documentclass[11pt]{amsart}

\usepackage{graphicx}
\usepackage{hyperref}
\usepackage[all]{xy}

\newtheorem{proposition}{Proposition}
\newtheorem{theorem}[proposition]{Theorem}

\newtheorem{corollary}[proposition]{Corollary}
\newtheorem{conjecture}[proposition]{Conjecture}
\newtheorem{question}[proposition]{Question}

\theoremstyle{definition}
\newtheorem{definition}[proposition]{Definition}

\theoremstyle{remark}
\newtheorem{remark}[proposition]{Remark}
\newtheorem{example}[proposition]{Example}
\newtheorem{exercise}[proposition]{Exercise}

\numberwithin{proposition}{section}

\def\A{\mathcal{A}}
\def\M{\mathcal{M}}
\def\d{\partial}
\def\Z{\mathbb{Z}}
\def\C{\mathbb{C}}
\def\R{\mathbb{R}}
\def\Aug{\operatorname{Aug}}
\def\AA{\mathbf{A}}
\def\Ahat{\hat{\mathbf{A}}}
\def\Acheck{\check{\mathbf{A}}}
\def\BB{\mathbf{B}}
\def\Bhat{\hat{\mathbf{B}}}
\def\Bcheck{\check{\mathbf{B}}}
\def\CC{\mathbf{C}}
\def\DD{\mathbf{D}}
\def\EE{\mathbf{E}}
\def\FF{\mathbf{F}}
\def\LL{\mathbf{\Lambda}}

\def\Aug{\operatorname{Aug}}
\def\diag{\operatorname{diag}}
\def\Phil{\mathbf{\Phi}^L}
\def\Phir{\mathbf{\Phi}^R}
\def\Aut{\operatorname{Aut}}

\begin{document}

\title{A Topological Introduction to Knot Contact Homology}
\author[L. Ng]{Lenhard Ng}
\address{Mathematics Department, Duke University, Durham, NC 27708 USA}
\email{ng@math.duke.edu}
\urladdr{\url{http://www.math.duke.edu/~ng/}}

\begin{abstract}
This is a survey of knot contact homology, with an emphasis on topological,
algebraic, and combinatorial aspects.
\end{abstract}

\maketitle

\section{Introduction}

This article is intended to serve as a general introduction to the subject of knot contact homology. There are two related sides to the theory: a geometric side devoted to the contact geometry of conormal bundles and explicit calculation of holomorphic curves, and an algebraic, combinatorial side emphasizing ties to knot theory and topology. We will focus on the latter side and only treat the former side lightly. The present notes grew out of lectures given at the Contact and Symplectic Topology Summer School in Budapest in July 2012.

The strategy of studying the smooth topology of a smooth manifold via
the symplectic topology of its cotangent bundle is an idea that was
advocated by V.\ I.\ Arnold and has been extensively studied in
symplectic geometry in recent years. It is well-known that if $M$ is
smooth then $T^*M$ carries a natural symplectic structure, with
symplectic form $\omega = -d\lambda_{\text{can}}$, where
$\lambda_{\text{can}} \in \Omega^1(T^*M)$ is the Liouville form; the idea
then is to analyze $T^*M$ as a symplectic manifold to recover
topological data about $M$.

In recent years this strategy has been executed quite successfully by
examining Gromov-type moduli spaces of holomorphic curves on
$T^*M$. For instance, one can show that the symplectic structure on
$T^*M$ recovers homotopic information about $M$, as shown in various
guises by Viterbo \cite{Viterbo},
Salamon--Weber \cite{SalamonWeber}, and Abbondandolo--Schwarz \cite{AS}, who each prove some version of the following
result (where technical restrictions have been omitted for simplicity):

\begin{theorem}[\cite{Viterbo,SalamonWeber,AS}]
The Hamiltonian Floer homology of $T^*M$ is isomorphic to the singular
homology of the free loop space of $M$.
\end{theorem}

\noindent
Subsequent work has related certain additional Floer-theoretic
constructions on $T^*M$ to the Chas--Sullivan loop product and string
topology; see for example \cite{AS2,CiL}.

In a slightly different direction, M.\ Abouzaid has used holomorphic
curves to show that the symplectic structure on $T^*M$ can contain
more than topological information about $M$:

\begin{theorem}[\cite{Abouzaid}]
If $\Sigma$ is an exotic $(4k+1)$-sphere that does not bound a
parallelizable manifold, then $T^*\Sigma$ is not symplectomorphic to
$T^*S^{4k+1}$.
\end{theorem}

\noindent
At the time of this writing, it is still possible that the smooth type
of a closed smooth manifold $M$ (up to diffeomorphism) is determined
by the symplectic type of $T^*M$ (up to symplectomorphism), which
would be a very strong
endorsement of Arnold's idea. (See however \cite{Knapp} for counterexamples
when $M$ is not closed.) For a nice discussion of
this and related problems, see \cite{Perutzsurvey}.

In this survey article, we discuss a relative version of Arnold's
strategy. The setting is as follows. Let $K \subset M$ be an embedded
submanifold (or an immersed submanifold with transverse
self-intersections). Then one can construct the \textit{conormal
  bundle} of $K$:
\[
L_K = \{(q,p)\,|\,q\in K,~\langle p,v\rangle = 0 \text{ for all }
v\in T_qK\} \subset T^*M.
\]
It is a standard exercise to check that $L_K$ is a Lagrangian
submanifold of $T^*M$.

One can work in one dimension lower by considering the \textit{cosphere} (unit
cotangent) \textit{bundle} $ST^*M$ of unit covectors in $T^*M$ with respect to
some metric; then $ST^*M$ is a contact manifold with contact form
$\alpha = \lambda_{\text{can}}$, and it can be shown that the contact structure
on $ST^*M$ is independent of the metric. The \textit{unit
  conormal bundle} of $K$,
\[
\Lambda_K = L_K \cap ST^*M \subset ST^*M,
\]
is then a Legendrian submanifold of $ST^*M$, with $\alpha|_{\Lambda_K} =
0$. See Figure~\ref{fig:conormal}.

\begin{figure}
\centerline{
\includegraphics[height=1.25in]{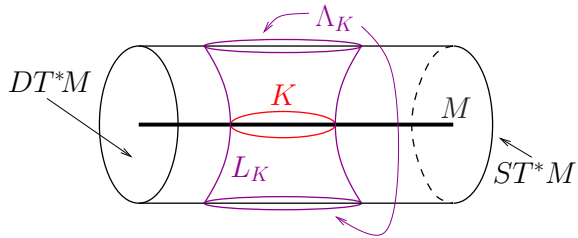}
}
\label{fig:conormal}
\caption{
A schematic depiction of cotangent and conormal bundles. Only the disk
bundle portion $DT^*M$ of $T^*M$ is shown, with boundary $ST^*M$. Note that
both $L_K$ and the zero section $M$ are Lagrangian in $T^*M$, and
their intersection is $K$.
}
\end{figure}

By construction, if $K$ changes by smooth isotopy in $M$, then
$\Lambda_K$ changes by Legendrian isotopy (isotopy within the class of Legendrian submanifolds) in $ST^*M$. One can then ask
what the Legendrian isotopy type of $\Lambda_K$ remembers about the
smooth isotopy type of $K$; see Question~\ref{q:Legendrian} below.

For the remainder of the section and article, we restrict our focus by
assuming that $M = \R^3$ and $K\subset \R^3$ is a knot or link. In
this case, $ST^*M$ is contactomorphic to the $1$-jet space $J^1(S^2) =
T^*S^2 \times \R$ equipped with the contact form
$dz-\lambda_{\text{can}}$, where $z$ is the coordinate on $\R$ and
$\lambda_{\text{can}}$ is the Liouville form on $S^2$, via the diffeomorphism
$ST^*\R^3 \to J^1(S^2)$ sending $(q,p)$
to $((p,q-\langle q,p\rangle p),\langle q,p\rangle)$ where $\langle
\cdot,\cdot\rangle$ is the standard metric on $\R^3$.

In the $5$-manifold $ST^*\R^3$, the unit conormal bundle $\Lambda_K$
is topologically a
$2$-torus (or a disjoint union of tori if $K$ has multiple
components). This can for instance be seen in the dual picture in $TR^3$, where the unit normal bundle can
be viewed as the boundary of a tubular neighborhood of $K$.
The topological type of
$\Lambda_K \cong T^2 \subset S^2\times\R^3$ contains no information: if $K_1,K_2$ are
arbitrary knots,
then $\Lambda_{K_1}$ and $\Lambda_{K_2}$ are smoothly
isotopic. (Choose a one-parameter family of possibly singular knots $K_t$ joining $K_1$ to $K_2$, and
perturb $\Lambda_{K_t}$ slightly when $K_t$ is singular to eliminate double points.)

However, there is no reason for $\Lambda_{K_1}$ and $\Lambda_{K_2}$ to be Legendrian
isotopic. This suggests the following question.

\begin{question}
How much of the topology of $K\subset\R^3$ is encoded in the
Legendrian structure of $\Lambda_K \subset ST^*\R^3$?
\label{q:Legendrian}
If
$\Lambda_{K_1}$ and $\Lambda_{K_2}$ are Legendrian isotopic, are $K_1$
and $K_2$ necessarily smoothly isotopic knots?
\end{question}

At the present, the answer to the second part of this question is
unknown but could possibly be ``yes''. The answer is known to be
``yes'' if either knot is the unknot; see below.

In order to tackle Question~\ref{q:Legendrian}, it is useful to have
invariants of Legendrian submanifolds under Legendrian isotopy. One
particularly powerful invariant is Legendrian contact homology, which
is a Floer-theoretic count of holomorphic curves associated to a
Legendrian submanifold and is
discussed in more detail in Section~\ref{sec:LCH}.

\begin{definition}
Let $K \subset \R^3$ be a knot or link. The \textit{knot contact
  homology} of $K$, written $HC_*(K)$, is the Legendrian contact
homology of $\Lambda_K$.
\end{definition}

\noindent
Knot contact homology is the homology of a differential graded algebra associated to a knot, the \textit{knot DGA} $(\A,\d)$.
By the general invariance result for Legendrian contact homology, the knot DGA and knot contact homology are topological
invariants of knots and links.

This article is a discussion of knot contact homology and its
properties. Despite the fact that the original definition of knot
contact homology involves holomorphic curves, there is a purely
combinatorial 
formulation of knot contact
homology. The article \cite{EENS}, which does most of the heavy lifting
for the results presented here, derives this combinatorial formula and
can be viewed as the first reasonably involved computation
of Legendrian contact homology in high dimensions.

Viewed from a purely knot theoretic perspective, knot contact homology
is a reasonably strong knot invariant. For instance, it detects the
unknot (see Corollaries~\ref{cor:unknot} and~\ref{cor:unknot2}): if
$K$ is a knot such that $HC_*(K) \cong HC_*(O)$ where $O$
is the unknot, then $K=O$. This implies in particular that the answer to Question~\ref{q:Legendrian} is yes if one of the knots is unknotted.
It is currently an open question whether
knot contact homology is a complete knot invariant.

Connections between knot contact homology and other knot invariants
are gradually beginning to appear. It is known that $HC_*(K)$
determines the Alexander polynomial
(Theorem~\ref{thm:knotDGAproperties}).  A portion of the
homology also has a natural topological interpretation, via an object called the
\textit{cord algebra} that is closely related to string topology.
In addition, one can use $HC_*(K)$ to define a
three-variable knot invariant, the \textit{augmentation polynomial},
which is closely related to the $A$-polynomial and conjecturally
determines a specialization of the HOMFLY-PT polynomial. Very
recently, a connection between knot contact homology and string theory
has been discovered, and this suggests that the augmentation
polynomial may in fact determine many known knot invariants, including the HOMFLY-PT polynomial and certain knot homologies, and may also be determined by a recursion relation for colored HOMFLY-PT polynomials.

Knot contact homology also produces a strong invariant of transverse
knots, which are knots that are transverse to the standard contact structure on $\R^3$.
For a transverse knot, the knot contact homology of the underlying
topological knot contains an additional filtered structure, \textit{transverse homology},
which is invariant under transverse isotopy. This has been shown to be an
effective transverse invariant (Theorem~\ref{thm:effective}), one of
two that are currently known
(the other comes from Heegaard Floer theory).

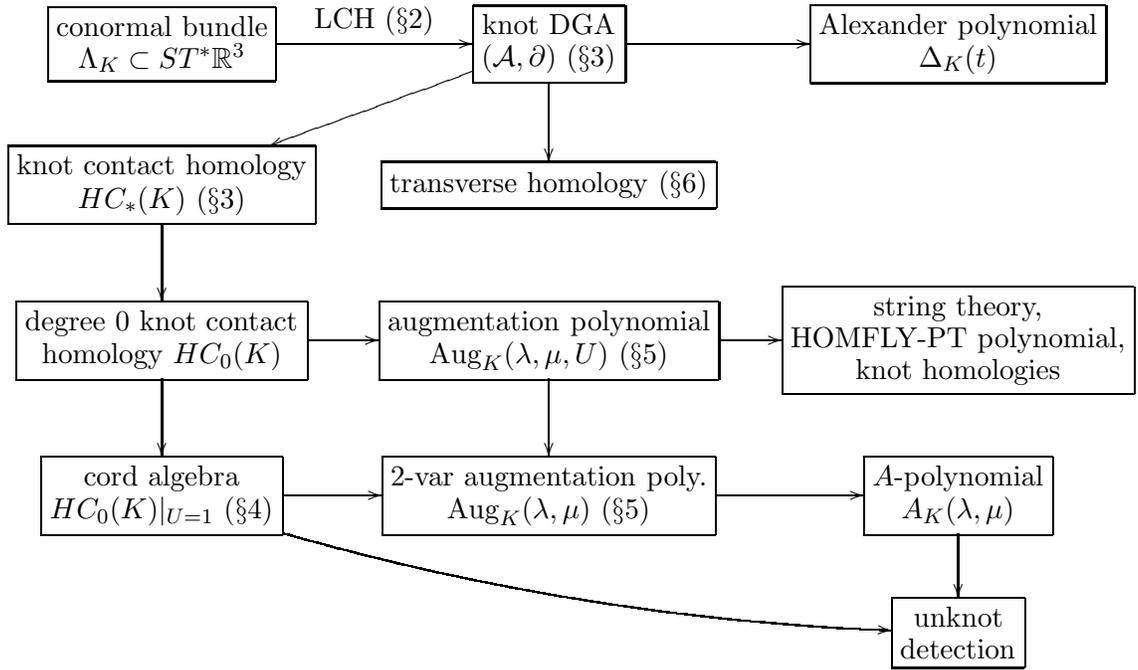
\begin{figure}
\xymatrix{
*+[F]\txt{conormal bundle \\$\Lambda_K \subset ST^*\R^3$}
\ar[r]^-{\txt{LCH (\S 2)}} & *+[F]\txt{knot
DGA\\ $(\A,\d)$ (\S 3)}
\ar[d] \ar[r]
\ar[dl] & *+[F]\txt{Alexander polynomial\\$\Delta_K(t)$}
\\
*+[F]\txt{knot contact homology\\$HC_*(K)$ (\S 3)} \ar[d]  &
*+[F]\txt{transverse homology (\S 6)} & \\
*+[F]\txt{degree $0$ knot contact\\homology $HC_0(K)$} \ar[r] \ar[d] &
*+[F]\txt{augmentation polynomial\\$\Aug_K(\lambda,\mu,U)$ (\S 5)} \ar[d]
\ar[r] &
*+[F]\txt{string theory,\\ HOMFLY-PT polynomial,\\knot homologies} \\
*+[F]\txt{cord algebra\\$HC_0(K)|_{U=1}$ (\S 4)} \ar[r] \ar@/_1pc/[drr]
& *+[F]\txt{2-var augmentation poly.\\ $\Aug_K(\lambda,\mu)$ (\S 5)} \ar[r] &
*+[F]\txt{$A$-polynomial\\$A_K(\lambda,\mu)$} \ar[d]\\
&&
*+[F]\txt{unknot\\detection} \\
}
\caption{
The knot invariants and interconnections described in this article.
}
\label{fig:plan}
\end{figure}
In the rest of the article, we expand on the properties of knot contact
homology mentioned above; see Figure~\ref{fig:plan} for a schematic
chart. In Section~\ref{sec:LCH}, we review the
general definition of Legendrian contact homology. We apply this to
knots and conormal bundles in Section~\ref{sec:kch} to give a
combinatorial definition of knot contact homology and present a few of
its properties. In Section~\ref{sec:cordalg}, we discuss the cord
algebra, which gives a topological interpretation of knot contact
homology in degree $0$. Section~\ref{sec:aug} defines the augmentation
polynomial and relates it to other knot invariants; this includes a
speculative discussion of the relation to string theory. In
Section~\ref{sec:transhom}, we present transverse homology and
consider its effectiveness as an invariant of transverse knots. Some
technical details (a definition of the ``fully noncommutative''
version of knot contact homology, and a comparison of the conventions
used in this article to conventions in the literature) are included in
the Appendix.

As this is a survey article, many details will be omitted in favor of
what we hope is an accessible exposition of the subject. (For more
introductory material on knot contact homology, the reader is referred
to two papers \cite{EE,NgsurveyBIRS}; note however that these
do not contain recent developments.)
There are
exercises scattered through the text as a concrete, hands-on
complement to the main discussion. There is not much new
mathematical content in this article beyond what has already appeared in
the literature, particularly \cite{EENStransverse,EENS} on the
geometric side and \cite{NgKCH1,NgKCH2,Ngframed,transhom} on the
combinatorial/topological side.
One exception is a
representation-theoretic interpretation of some factors of the augmentation
polynomial that do not appear in the $A$-polynomial; see Theorem~\ref{thm:rep}. We have also introduced a number of conventions
for combinatorial knot contact homology in this article that are new and,
in the author's opinion, more natural than previous conventions.

\subsection*{Acknowledgments}

I am grateful to the organizers and
participants of the 2012 Contact and Symplectic Topology Summer School
in Budapest, and particularly Chris Cornwell and Tobias Ekholm, for a great deal of helpful feedback, and to Dan Rutherford for catching a typo in an earlier draft. This work was
partially supported by NSF CAREER grant DMS-0846346.

\section{Legendrian Contact Homology}
\label{sec:LCH}

In this section, we give a cursory introduction to Legendrian
contact homology and augmentations, essentially the minimum necessary
to motivate the construction of knot contact homology in
Section~\ref{sec:kch}. The reader interested in further details is
referred to the various references given in this section.

Legendrian contact homology (LCH), introduced by Eliashberg and Hofer
in \cite{Eli98}, is an invariant of Legendrian submanifolds in
suitable contact manifolds. This invariant is defined by counting
certain holomorphic curves in the symplectization of the contact
manifold, and is a part of the (much larger) Symplectic Field Theory package of
Eliashberg, Givental, and Hofer \cite{EGH}. LCH is the homology of a
differential graded algebra (DGA) that we now describe, and in some
sense the DGA (up to an appropriate equivalence relation), rather than
the homology, is the ``true'' invariant of the Legendrian submanifold.

In this section, we will work exclusively in a contact manifold of the
form $V = J^1(M) = T^*M \times\R$ with the standard contact form
$\alpha = dz-\lambda_{\text{can}}$. LCH can be defined for much more
general contact manifolds, but the proof of invariance in general has
not been fully carried out, and even the definition is more
complicated than the one given below when the contact manifold has
closed Reeb orbits. Note that for $V = J^1(M)$, the Reeb vector field
$R_{\alpha}$ is $\d/\d z$ and thus $J^1(M)$ has no closed Reeb orbits.

Let $\Lambda \subset V$ be a Legendrian submanifold. We assume for
simplicity that $\Lambda$ has trivial Maslov class (e.g., for
Legendrian knots in $\R^3 = J^1(\R)$, this means that $\Lambda$ has
rotation number $0$), and that $\Lambda$ has finitely many Reeb
chords, integral curves for the Reeb field $R_\alpha$ with endpoints
on $\Lambda$. We
label the Reeb chords formally as $a_1,\ldots,a_n$. Finally, let $R$
denote (here and
throughout the article) the coefficient ring $R = \Z [H_2(V,\Lambda)]$, the group ring of the relative homology group $H_2(V,\Lambda)$.

\begin{definition}
The LCH \textit{differential graded algebra} associated to $\Lambda$
is $(\A,\d)$, defined as follows: \label{def:LCHDGA}
\begin{enumerate}
\item
\textbf{Algebra}: $\A = R \langle a_1,\ldots,a_n\rangle$ is the free
noncommutative unital algebra over $R$ generated by $a_1,\ldots,a_n$.
As an $R$-module, $\A$ is generated by all words $a_{i_1}\cdots
a_{i_k}$ for $k\geq 0$ (where $k=0$ gives the empty word $1$).
\item
\textbf{Grading}: Define $|a_i| = CZ(a_i)-1$, where $CZ$ denotes
Conley--Zehnder index (see \cite{EES07} for the definition in this context)
and $|r|=0$ for $r\in R$. Extend the grading to all of $\A$ in the usual way:
$|xy| = |x|+|y|$.
\item
\textbf{Differential}: Define $\d(r) =0$ for $r\in R$ and
\[
\d(a_i) = \sum_{\dim \M(a_i;a_{j_1},\ldots,a_{j_k})/\R = 0} ~~\sum_{\Delta\in\M/\R} (\textrm{sgn}(\Delta)) e^{[\Delta]} a_{j_1}\cdots a_{j_k}
\]
where $\M(a_i;a_{j_1},\ldots,a_{j_k})$ is the moduli space defined
below, $\text{sgn}(\Delta)$ is an orientation sign associated to
$\Delta$,
and $[\Delta]$ is the homology class\footnote{To define this homology
  class, we assume that ``capping half-disks'' have been chosen in $V$
  for each Reeb chord $a_i$, with boundary given by $a_i$ along with a
  path in $\Lambda$ joining the endpoints of $a_i$. Some additional
  care must be taken if $\Lambda$ has multiple components.}
 of $\Delta$ in $H_2(V,\Lambda)$.

Extend the differential to all of $\A$ via the signed Leibniz rule: $\d(xy) = (\d x)y+(-1)^{|x|}x(\d y)$.
\end{enumerate}

\end{definition}

The key to Definition~\ref{def:LCHDGA} is the moduli space
$\M(a_i;a_{j_1},\ldots,a_{j_k})$. To define this, let $J$ be a
(suitably generic) almost
complex structure on the symplectization $(\R\times V,d(e^t\alpha))$
of $V$ (where $\alpha$ is the contact form on $V$ and $t$ is the $\R$
coordinate) that is compatible with the symplectization in the
following sense: $J$ is $\R$-invariant, $J(\d/\d t) = R_\alpha$, and
$J$ maps $\xi = \ker\alpha$ to itself. With respect to this almost
complex structure, $\R\times a_i$ is a holomorphic strip for any Reeb
chord $a_i$ of $\Lambda$.

\begin{figure}
\centerline{
\includegraphics[height=2.25in]{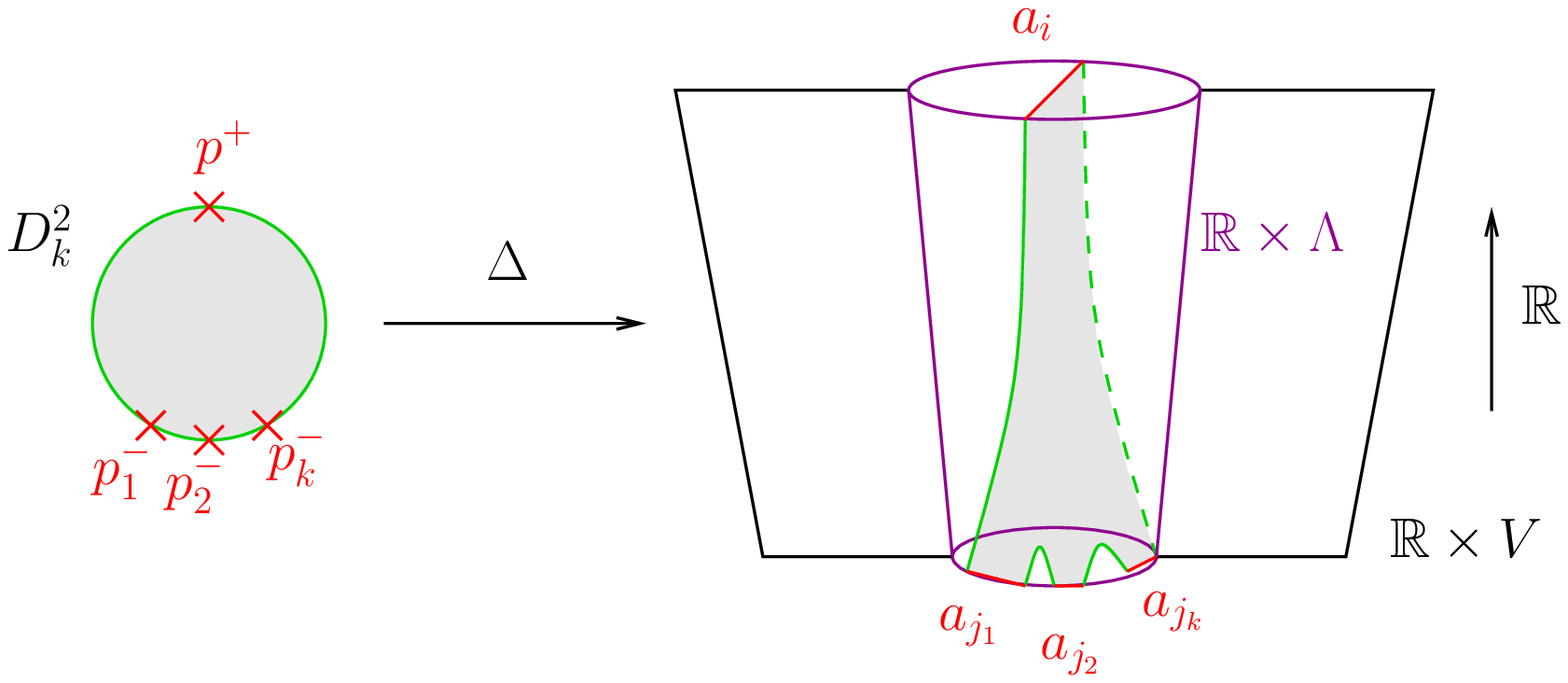}
}
\label{fig:holdisk}
\caption{
A holomorphic disk $\Delta :\thinspace (D^2_k,\d D^2_k) \to (\R\times
V,\R\times\Lambda)$ contributing to $\M(a_i;a_{j_1,\ldots,a_{j_k}})$
and the differential $\d(a_i)$.
}
\end{figure}

Let $D^2_k = D^2\setminus \{p^+,p_1^-,\ldots,p_k^-\}$ be a closed disk
with $k+1$ punctures on its boundary, labeled $p^+,p_1^-,\ldots,p_k^-$
in counterclockwise order around $\d D^2$.
For (not necessarily distinct) Reeb
chords $a_i$ and $a_{j_1},\ldots,a_{j_k}$ for some $k\geq 0$, let
$\M(a_i;a_{j_1},\ldots,a_{j_k})$ be the moduli space of
$J$-holomorphic maps
\[
\Delta :\thinspace (D^2_k,\d D^2_k) \to (\R\times V,\R\times\Lambda)
\]
up to domain reparametrization, such that:
\begin{itemize}
\item
near $p^+$, $\Delta$ is asymptotic to
a neighborhood of the Reeb strip $\R\times a_i$ near $t=+\infty$;
\item
near $p_l^-$ for $1\leq l\leq k$, $\Delta$ is asymptotic to a
neighborhood of $\R\times a_{j_l}$ near $t=-\infty$.
\end{itemize}
See Figure~\ref{fig:holdisk}.

When everything is suitably generic, $\M(a_i;a_{j_1},\ldots,a_{j_k})$
is a manifold of dimension $|a_i|-\sum_l |a_{j_l}|$. The moduli space
also has an $\R$ action given by translation in the $\R$ direction,
and the differential $\d(a_i)$ counts moduli spaces
$\M(a_i;a_{j_1},\ldots,a_{j_k})$ that are rigid after quotienting by
this $\R$ action.

\begin{remark}
If $H_2(V,\Lambda) \cong H_2(V) \oplus H_1(\Lambda)$, as is true in
the case that we will consider, one can ``improve'' the DGA $(\A,\d)$
to a DGA that we might call the \textit{fully noncommutative DGA}
$(\widetilde{\A},\d)$, defined as follows.
\label{rmk:noncommDGA}
For simplicity, assume that $\Lambda$ is connected; there is a similar but slightly more involved construction otherwise.
The algebra
$\widetilde{\A}$ is the tensor algebra over the group ring $\Z[H_2(V)]$, generated by the
Reeb chords $a_1,\ldots,a_n$ along with elements of $\pi_1(\Lambda)$, with no relations except for the ones
inherited from $\pi_1(\Lambda)$. Thus $\widetilde{\A}$ is generated as a
$\Z[H_2(V)]$-module by words of the form
\[
\gamma_0 a_{i_1} \gamma_1 a_{i_2} \gamma_2 \cdots \gamma_{k-1} a_{i_k}
\gamma_k
\]
where $a_{i_1},\ldots,a_{i_k}$ are Reeb chords of $\Lambda$,
$\gamma_0,\ldots,\gamma_k \in \pi_1(\Lambda)$, and $k\geq 0$. Note that
$\A$ is a quotient of $\widetilde{\A}$: just abelianize $\pi_1(\Lambda)$ to $H_1(\Lambda)$, and allow Reeb chords $a_i$ to
commute with homology classes $\gamma\in H_1(\Lambda)$.

To define the differential, let $\Delta$ be a disk in
$\M(a_i;a_{j_1},\ldots,a_{j_k})$. The projection map $\pi :\thinspace
H_2(V,\Lambda) \to H_2(V)$ gives a class $\pi([\Delta]) \in
H_2(V)$. The boundary of the image of $\Delta$ consists of an ordered
collection of $k+1$ paths in $\Lambda$ joining endpoints of Reeb
chords. By fixing paths in $\Lambda$ joining each Reeb chord endpoint
to a fixed point on $\Lambda$, one can
close these $k+1$ paths into $k+1$ loops in $\Lambda$. Let
$\gamma_0(\Delta),\ldots,\gamma_k(\Delta)$ denote the homotopy classes of these loops
in $\pi_1(\Lambda)$, where the loops are ordered in the order that they
appear in the image of $\d D^2$, traversed counterclockwise.
Finally, define $\d(\gamma) = 0$ for $\gamma \in \pi_1(\Lambda)$ and
\[
\d(a_i) = \sum_{\dim \M(a_i;a_{j_1},\ldots,a_{j_k})/\R = 0}
~~\sum_{\Delta\in\M/\R} (\textrm{sgn}(\Delta)) e^{\pi([\Delta])}
\gamma_0(\Delta) a_{j_1}
\gamma_1(\Delta)\cdots a_{j_k} \gamma_k(\Delta),
\]
and extend the differential to $\widetilde{\A}$ by the Leibniz rule.

Note that the quotient that sends $\widetilde{\A}$ to $\A$ also sends
the differential on $\widetilde{\A}$ to the differential on $\A$. The
fully noncommutative DGA $(\widetilde{\A},\d)$ satisfies the same
properties as $(\A,\d)$ (Theorem~\ref{thm:LCHinv} below), with a
suitable alteration of the definition of stable tame isomorphism. For
the majority of this article, we will stick to the usual LCH DGA
$(\A,\d)$, which is enough for most purposes, because it simplifies
notation; see however the discussion after Theorem~\ref{thm:htpy}, as
well as the Appendix.
\end{remark}

We now state some fundamental properties of the LCH DGA
$(\A,\d)$. These began with the work of Eliashberg--Hofer
\cite{Eli98}; Chekanov \cite{Ch02} wrote down the precise statement
and gave a combinatorial proof for the case $V = \R^3$ (see also
\cite{ENS}). The formulation given here is due to, and proven by,
Ekholm--Etnyre--Sullivan \cite{EES07}.

\begin{theorem}[\cite{Eli98,Ch02,EES07}]
Given suitable genericity assumptions: \label{thm:LCHinv}
\begin{enumerate}
\item
$\d$ decreases degree by 1;
\item
$\d^2=0$;
\item \label{item:3}
up to stable tame isomorphism, $(\A,\d)$ is independent of all choices
(of contact form for the contact structure on $V$, and of $J$), and is
an invariant of $\Lambda$ up to Legendrian isotopy;
\item \label{item:4}
up to isomorphism,
$H_*(\A,\d) =: HC_*(V,\Lambda)$ is also an invariant of $\Lambda$ up to
Legendrian isotopy.
\end{enumerate}
\end{theorem}

\noindent
Here ``stable tame isomorphism'' is an equivalence relation between
DGAs defined in Definition~\ref{def:sti} below, which is a special
case of quasi-isomorphism; thus item~\ref{item:3} in
Theorem~\ref{thm:LCHinv} directly implies item~\ref{item:4}. The
homology $HC_*(V,\Lambda)$ is called the \textit{Legendrian contact
  homology} of $\Lambda$.

\begin{definition}[\cite{Ch02}, see also \cite{ENS}]
\label{def:sti}
\begin{enumerate}
\item
Let $\A = R\langle a_1,\ldots,a_n\rangle$. An \textit{elementary automorphism} of $\A$ is an algebra map $\phi :\thinspace \A\to\A$ of the form: for some $i$,
$\phi(a_j) = a_j$ for all $j\neq i$, and $\phi(a_i) = a_i+v$ for some
$v\in R\langle a_1,\ldots,a_{i-1},a_{i+1},\ldots,a_n\rangle$.
\item
A \textit{tame automorphism} of $\A$ is a composition of elementary
automorphisms.
\item
DGAs $(\A=R\langle a_1,\ldots,a_n\rangle,\d)$ and $(\A'=R\langle a_1',\ldots,a_n'\rangle,\d')$ are \textit{tamely isomorphic} if  there is an algebra isomorphism $\psi = \phi_2 \circ \phi_1$ such that $\phi_1 :\thinspace \A\to\A$ is a tame automorphism and $\phi_2 :\thinspace \A\to\A'$ is given by $\phi_2(a_i)=a_i'$ for all $i$, and $\psi$ intertwines the differentials:
$\psi \circ \d = \d' \circ \psi$.
\item
A \textit{stabilization} of $(\A=R\langle a_1,\ldots,a_n\rangle,\d)$
is $(S(\A),\d)$, where $S(\A)=R\langle a_1,\ldots,a_n,e_1,e_2\rangle$
with grading inherited from $\A$ along with $|e_1|=|e_2|+1$, and $\d$ is
induced on $S(\A)$ by $\d$ on $\A$ along with $\d(e_1)=e_2$,
$\d(e_2)=0$.
\item
DGAs $(\A,\d)$ and $(\A',\d')$ are \textit{stable tame isomorphic} if they are tamely isomorphic after stabilizing each of them some (possibly different) number of times.
\end{enumerate}
\end{definition}

\begin{exercise}
\label{exc:sti}
\begin{enumerate}
\item
Prove that $H(S(\A),\d) \cong H(\A,\d)$ and thus stable tame
isomorphism implies quasi-isomorphism.
\item
Prove that if $(\A,\d)$ is a DGA with a generator $a$ satisfying
$|a|=1$ and $\d(a)=1$, then $H(\A,\d) = 0$. Conclude that
quasi-isomorphism does not necessarily imply stable tame isomorphism.
\item
If all generators of $\A$ are in degree $\geq 0$, and $S$ is a unital
ring, show that there is a one-to-one correspondence between
augmentations of $(\A,\d)$ to $S$ (see
Definition~\ref{def:augmentation} below) and ring homomorphisms $H_0(\A,\d)
\to S$. Find an example to show that this is not true in general
without the degree condition.
\item \label{item:unknotsti}
Find the stable tame isomorphism in Example~\ref{ex:unknot} below.
\end{enumerate}
\end{exercise}

We conclude this section by introducing the notion of an augmentation,
which is an important algebraic tool for studying DGAs.

\begin{definition}
Let $(\A,\d)$ be a DGA over $R$, and let $S$ be a unital ring.
\label{def:augmentation}
An \textit{augmentation} of $(\A,\d)$ to $S$ is a graded ring homomorphism
\[
\epsilon :\thinspace \A\to S
\]
sending $\d$ to $0$; that is, $\epsilon \circ \d = 0$, $\epsilon(1)=1$, and $\epsilon(a)=0$ unless $|a|=0$.
\end{definition}

Note that augmentations use the multiplicative structure on the DGA
$(\A,\d)$. An augmentation allows one to construct a linearized
version of the homology of $(\A,\d)$.

\begin{exercise}
\label{exc:lin}
Let $(\A,\d)$ be the LCH DGA for a Legendrian $\Lambda$, and let
$\epsilon$ an augmentation of $(\A,\d)$ to $S$.
\begin{enumerate}
\item
Write $\A = R\langle a_1,\ldots,a_n\rangle$. The augmentation
$\epsilon$ induces an augmentation
$\epsilon_S :\thinspace S\langle a_1,\ldots,a_n\rangle \to S$ that
acts as the identity on $S$ and as $\epsilon$ on the $a_i$'s.
Prove that $(\ker\epsilon_S) / (\ker\epsilon_S)^2$ is a finitely
generated, graded $S$-module.
\item
Prove that $\d$ descends to a map here: then
\[
HC_*^{\text{lin}}(\Lambda,\epsilon) := H_*((\ker\epsilon) / (\ker\epsilon)^2,\d)
\]
is a graded $S$-module, the \textit{linearized Legendrian contact
  homology} of $\Lambda$ with respect to the augmentation $\epsilon$.
\end{enumerate}
\end{exercise}

\begin{remark}
Here is a less concise, but possibly more illuminating, description of
linearized contact homology. We can define a differential $\d_S$ on
$\A_S := S\langle a_1,\ldots,a_n\rangle$ by composing $\d$ by the map
$R\to S$ induced by $\epsilon$ (this map fixes all $a_i$'s).
Define an $S$-algebra automorphism $\phi_\epsilon :\thinspace \A_S \to
\A_S$ by $\phi_\epsilon(a_i) = a_i + \epsilon(a_i)$ for all $i$ and
$\phi_\epsilon(s) = s$ for all $s\in S$. Then the map
\[
\d_{S,\epsilon} := \phi_\epsilon \circ \d_S \circ \phi_\epsilon^{-1}
 \]
 is a differential on $\A_S$. Furthermore, if we define $\A_S^+$ to be
 the subalgebra of $\A_S$ generated by $a_1,\ldots,a_n$, so that $\A_S
 \cong S \oplus A_S^+$ as $S$-modules, then
$\d_{S,\epsilon}$ restricts to a map from $\A_S^+$ to itself, and so
it induces a differential from $\A_S^+/(\A_S^+)^2$ to itself. The
homology of the complex $(\A_S^+/(\A_S^+)^2,\d_{S,\epsilon})$ is the
linearized contact homology of $\Lambda$ with respect to $\epsilon$.
\end{remark}

\begin{remark}
Let $\Lambda \subset V$ have LCH DGA $(\A,\d)$, and write $R = \Z[H_2(V,\Lambda)]$ as usual.
\label{rmk:augvar}
Any augmentation $\epsilon$ of $(\A,\d)$ to a ring $S$
induces a map $\epsilon|_R
:\thinspace R \to S$, since $R \subset \A$. This motivates the
following definition: define the \textit{augmentation variety} of
$\Lambda$ to $S$ to be
\begin{align*}
\Aug(\Lambda,S) &= \{\varphi :\thinspace R\to S \,|\,
\varphi = \epsilon|_R \text{ for some augmentation } \epsilon
\text{ from } (\A,\d) \text{ to } S \}\\
& \qquad \subset \operatorname{Hom}(R,S).
\end{align*}
It follows from Theorem~\ref{thm:LCHinv} that $\Aug(\Lambda,S)$ is an
invariant of $\Lambda$ under Legendrian isotopy.

In the simplest case, when $V = \R^3$ and $\Lambda$ is a Legendrian
knot, one can consider the augmentation variety
\[
\Aug(\Lambda,S) \subset \operatorname{Hom}(\Z[\Z],S) \cong S^\times
\]
where $S^\times$ is the multiplicative group of units in $S$. It can
then be shown (by upcoming work of Caitlin Leverson) that
$\Aug(\Lambda,S)$ is either $\{-1\}$ if $\Lambda$ has a (graded)
ruling, or $\emptyset$ otherwise; the augmentation variety contains
fairly minimal information about $\Lambda$. However, in the main case
of interest in this article, where $V = J^1(S^2)$ and $\Lambda =
\Lambda_K$, the augmentation variety contains a great deal of
information about $\Lambda_K$. See Section~\ref{sec:aug}.
\end{remark}

\begin{remark}
A geometric motivation for augmentations comes from exact Lagrangian
fillings. Here is a somewhat imprecise description.
Suppose that the contact manifold $V$ is a convex end of an open exact
symplectic manifold $(W,\omega)$; for instance, $W$ could be the
symplectization of $V$, or an exact symplectic filling of $V$. Let $L
\subset W$ be an oriented exact Lagrangian submanifold whose boundary is the
Legendrian $\Lambda \subset V$. Then $L$ induces an augmentation
$\epsilon$ of the LCH DGA of $\Lambda$, to the ring
$S = \Z[H_2(W,L)]$, which restricts on the coefficient ring to the
usual map $\Z[H_2(V,\Lambda)] \to \Z[H_2(W,L)]$. This augmentation is
defined as follows: $\epsilon(a_i)$ is the sum of all rigid
holomorphic disks in $W$ with boundary on $L$ and positive boundary
puncture limiting to the Reeb chord $a_i$ of $\Lambda$, where each
holomorphic disk contributes its homology class in $H_2(W,L)$. The
fact that $\epsilon$ is an augmentation is established by an argument
similar to the proof that $\d^2=0$ in Theorem~\ref{thm:LCHinv} above,
which involves two-story holomorphic buildings.
\end{remark}

\section{Knot Contact Homology}
\label{sec:kch}

In this section, we present a combinatorial calculation of knot
contact homology, which is Legendrian contact
homology in the particular case where the contact manifold
is $ST^*\R^3 \cong J^1(S^2)$ and the Legendrian submanifold is the
unit conormal bundle $\Lambda_K$ to some link $K \subset \R^3$. The
version of knot contact homology we give here is a theory over
the coefficient ring $\Z[\lambda^{\pm 1},\mu^{\pm 1},U^{\pm 1}]$, and
has appeared in the literature in several places and guises,\footnote{The profusion of terms and specializations is an
unfortunate byproduct of the way that the subject evolved over a decade.} up to
various changes of variables (see the Appendix). Our presentation
corresponds to what is called the ``infinity'' version of transverse
homology in \cite{EENStransverse,transhom}, and is the most general
(as of now) version of knot contact homology for topological knots and links.
Setting $U=1$, one obtains an invariant called ``framed knot
contact homology'' in \cite{Ngframed} and simply ``knot contact
homology'' in \cite{EENS}. If we set $U=\lambda=1$ and $\mu=-1$, we
obtain the original version of knot contact homology from
\cite{NgKCH1,NgKCH2}.

For simplicity, we assume that $K \subset \R^3$ is an oriented knot; see
Remark~\ref{rmk:link} below for the case of a multi-component
link. The unit conormal
bundle $\Lambda_K \subset
J^1(S^2)$ is a Legendrian $T^2$.
As discussed in the previous section, the LCH DGA of
$\Lambda_K$ is a topological link invariant. The coefficient ring for
this DGA is
\[
R = \Z[H_2(J^1(S^2),\Lambda_K)] \cong \Z[\lambda^{\pm 1},\mu^{\pm
  1},U^{\pm 1}],
\]
where $\lambda,\mu$ correspond to the longitude and meridian generators of
$H_1(\Lambda_K)$ and $U$ corresponds to the generator of
$H_2(J^1(S^2)) = H_2(S^2)$. Note that the choice of $\lambda,\mu$
relies on a choice of (orientation and) framing for $K$; we choose the
Seifert framing for definiteness.

\begin{definition}
$K\subset\R^3$ knot. The \textit{knot DGA} of $K$ is the LCH
differential graded algebra of $\Lambda_K \subset J^1(S^2)$, an
algebra over the ring $R=\Z[\lambda^{\pm 1},\mu^{\pm
  1},U^{\pm 1}]$.
\label{def:KCH}
The homology of this DGA is the \textit{knot contact homology} of $K$,
$HC_*(K) = HC_*(ST^*\R^3,\Lambda_K)$.
\end{definition}

\begin{remark}
If $K$ is an oriented $r$-component link, one can similarly define the ``knot
DGA'', now an algebra over
\[
\Z[H_2(J^1(S^2),\Lambda_K)] \cong \Z[\lambda_1^{\pm
  1},\ldots,\lambda_r^{\pm 1},\mu_1^{\pm
  1},\ldots,\mu_r^{\pm 1},U^{\pm 1}].
\]
\label{rmk:link}
Here, as in the knot case, we choose the $0$-framing on each link
component to fix the above isomorphism. The combinatorial description
for the DGA in the link case is a bit more involved than for the knot
case; see the Appendix for details.
\end{remark}

We now return to the case where $K$ is a knot.
It follows directly from Theorem~\ref{thm:LCHinv} that knot contact
homology $HC_*(K)$ is an invariant up to $R$-algebra isomorphism, as
is the knot DGA up to stable tame isomorphism. What we describe next is
a combinatorial form for the knot DGA, given a braid presentation of
$K$; this follows the papers \cite{EENStransverse,transhom}, which build on
previous work
\cite{EENS,NgKCH1,NgKCH2,Ngframed}. The fact that the
combinatorial DGA agrees with the
holomorphic-curve DGA described in Section~\ref{sec:LCH} is a rather
intricate calculation and the subject of \cite{EENS}.

Let $B_n$ be the braid group on $n$ strands. Define $\A_n$ to be the
free noncommutative unital algebra over $\Z$ generated by $n(n-1)$
generators $a_{ij}$ with $1\leq i,j\leq n$ and $i\neq j$. We consider
the following representation of $B_n$ as a group of automorphisms of
$\A_n$, which was first introduced (in a slightly different form) in
\cite{Magnus}.

\begin{definition}
The braid homomorphism $\phi :\thinspace B_n \to \operatorname{Aut} \A_n$ is the map defined on generators $\sigma_k$ ($1\leq k\leq n-1$) of $B_n$ by:
\[
\phi_{\sigma_k} :\thinspace
\begin{cases}
a_{ij} \mapsto a_{ij}, & i,j \neq k,k+1 \\
a_{k+1,i} \mapsto a_{ki}, & i\neq k,k+1 \\
a_{i,k+1} \mapsto a_{ik}, & i\neq k,k+1 \\
a_{k,k+1} \mapsto -a_{k+1,k} & \\
a_{k+1,k} \mapsto -a_{k,k+1} & \\
a_{ki} \mapsto a_{k+1,i}-a_{k+1,k}a_{ki}, & i\neq k,k+1 \\
a_{ik} \mapsto a_{i,k+1}-a_{ik}a_{k,k+1}, & i\neq k,k+1. \\
\end{cases}
\]
This extends to a map on $B_n$ (see the following exercise).
\end{definition}

\begin{exercise}
 \label{exc:homom}
\begin{enumerate}
\item
Check that $\phi_{\sigma_k}$ is invertible.
\item
Check that $\phi$ respects the braid relations: $\phi_{\sigma_k}\phi_{\sigma_{k+1}}\phi_{\sigma_k} = \phi_{\sigma_{k+1}}\phi_{\sigma_k}\phi_{\sigma_{k+1}}$ and $\phi_{\sigma_i}\phi_{\sigma_j} = \phi_{\sigma_j}\phi_{\sigma_i}$ for $|i-j| \geq 2$.
\item
For the braid $B = (\sigma_1\cdots \sigma_{n-1})^m \in B_n$ for $m
\geq 1$, calculate $\phi_B$. (The answer is quite simple.) \label{item:twist}
\end{enumerate}
\end{exercise}

\begin{remark}
As a special case of Exercise~\ref{exc:homom}(\ref{item:twist}), when
$B$ is a full twist $(\sigma_1\cdots \sigma_{n-1})^n$, $\phi_B$ is the
identity map; thus $\phi :\thinspace B_n \to \Aut\A_n$ is not a
faithful representation. However, one can create a faithful
representation of $B_n$ from $\phi$, as follows. Embed $B_n$ into
$B_{n+1}$ by adding an extra (noninteracting) strand to any braid in
$B_n$; then the composition
\[
B_n \hookrightarrow B_{n+1} \stackrel{\phi}{\to} \Aut\A_{n+1}
\]
is a faithful representation of $B_n$ as a group of algebra
automorphisms of $\A_{n+1}$.
\label{rmk:faithful}
See \cite{NgKCH2}.
\end{remark}

Before we proceed with the combinatorial definition of the knot DGA,
we present a possibly illustrative reinterpretation of $\phi$ that
begins by viewing $B_n$ as the mapping class group of $D^2 \setminus
\{p_1,\ldots,p_n\}$; this will be useful in
Section~\ref{sec:cordalg}. To this end, let $p_1,\ldots,p_n$ be a
collection of $n$ points in $D^2$, which we arrange in order in a
horizontal line.

\begin{definition}
An \textit{arc} is a continuous path $\gamma :\thinspace
[0,1] \to D^2$ such that $\gamma^{-1}(\{p_1,\ldots,p_n\}) = \{0,1\}$;
that is, the path begins at some $p_i$, ends at some $p_j$ (possibly
the same point), and otherwise does not pass through any of the
$p$'s. We consider arcs up to endpoint-fixing homotopy through arcs:
two arcs are identified if, except at their endpoints, they are
homotopic in $D^2 \setminus \{p_1,\ldots,p_n\}$. \label{def:arc}
Let $\widetilde{\A}$ denote the
tensor algebra over $\Z$ generated by arcs, modulo the (two-sided ideal generated by the) relations:
\begin{enumerate}
\item \label{item:arc1}
$(\raisebox{-0.15in}{\includegraphics[width=0.6in]{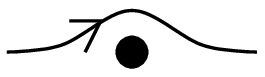}})
-
(\raisebox{-0.15in}{\includegraphics[width=0.6in]{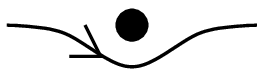}})
-
(\raisebox{-0.15in}{\includegraphics[width=0.6in]{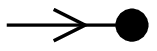}}) \cdot
(\raisebox{-0.15in}{\includegraphics[width=0.6in]{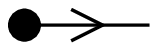}})
= 0$,
\newline where each of these dots indicates the same point $p_i$;
\item \label{item:arc2}
any contractible arc with both endpoints at some $p_i$ is equal to $0$.
\end{enumerate}
\end{definition}

\begin{remark}
There is a notion of a framed arc that generalizes
Definition~\ref{def:arc}, and a corresponding version of
$\widetilde{\A}$ in which $0$ is replaced by $1-\mu$.
\label{rmk:framedarc}
Framed arcs are
used to relate knot contact homology to the cord algebra (see
Section~\ref{sec:cordalg}), but we omit their definition here in the
interest of simplicity. See \cite{Ngframed} for more details.
\end{remark}

One can now relate the homomorphism $\phi$ with the algebra
$\widetilde{\A}$ generated by arcs.

\begin{theorem}[\cite{NgKCH2}]
\begin{enumerate}
\item
For $i\neq j$, let $\gamma_{ij}$ denote the arc depicted below (left
diagram for $i<j$, right for $i>j$):
\[
\includegraphics[width=3.5in]{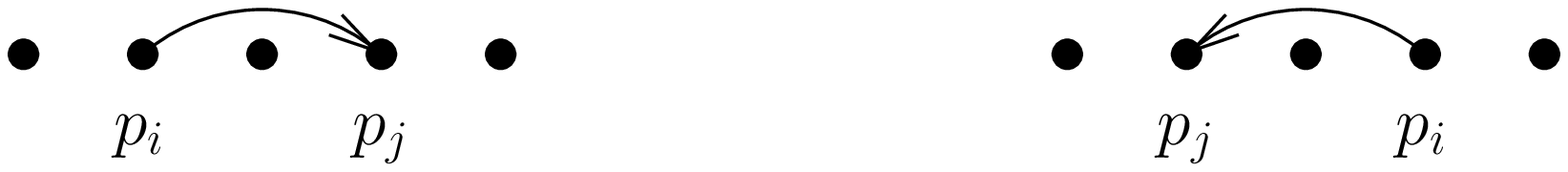}
\]
Then the map sending $a_{ij}$ to $\gamma_{ij}$ for $i<j$ and
$-\gamma_{ij}$ for $i>j$
induces an algebra isomorphism $\Phi :\thinspace \A_n
\stackrel{\cong}{\to} \widetilde{\A}$.
\item
For any $B \in B_n$ and any $i,j$, we have \label{item:arcs2}
\[
\Phi(\phi_B(a_{ij})) = B \cdot \Phi(a_{ij}),
\]
where $B$ acts on $\widetilde{\A}$ by the mapping class group action:
if $a$ is an arc, then $B \cdot a$ is the arc obtained by applying to
$a$ the
diffeomorphism of $D^2 \setminus \{p_1,\ldots,p_n\}$ given by $B$.
\end{enumerate} \label{thm:arcs}
\end{theorem}

As an illustration of Theorem~\ref{thm:arcs}(\ref{item:arcs2}), the
braid $B = \sigma_k$ sends the arc $\gamma_{ki}$ for $i>k+1$ to
\[
\raisebox{-0.18in}{\includegraphics[width=1in]{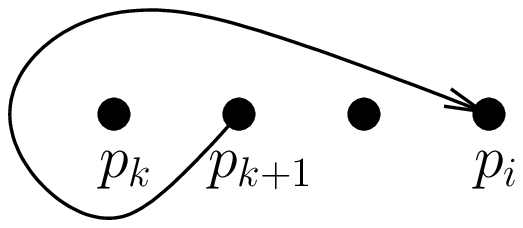}} =
\left(\raisebox{-0.18in}{\includegraphics[width=1in]{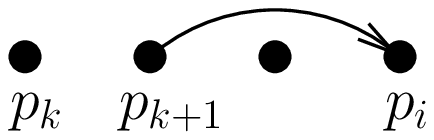}}\right)
+
\left(\raisebox{-0.18in}{\includegraphics[width=1in]{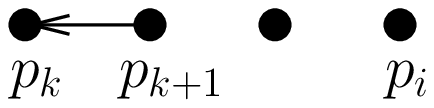}}\right)
\cdot
\left(\raisebox{-0.18in}{\includegraphics[width=1in]{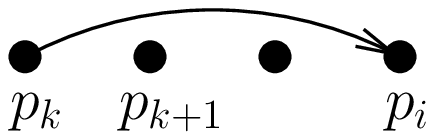}}\right),
\]
where the equality is in $\widetilde{\A}$ and uses the skein relation
in Definition~\ref{def:arc}; the right hand side is the image under
$\Phi$ of $a_{k+1,i}-a_{k+1,k}a_{ki} = \phi_{\sigma_k}(a_{ki})$.

We now proceed with the definition of the knot DGA. We will need two
$n\times n$ matrices
$\Phil_B,\Phir_B$ that arise from the representation $\phi$ (or, more
precisely, its extension as described in Remark~\ref{rmk:faithful}).

\begin{definition}[\cite{NgKCH1}] \label{def:PhiLPhiR}
Let $B\in B_n \hookrightarrow B_{n+1}$, and label the additional strand in $B_{n+1}$ by $*$. Define $\Phil_B,\Phir_B \in \operatorname{Mat}_{n\times n}(\A_n)$ by:
\begin{align*}
\phi_B(a_{i*}) &= \sum_{j=1}^n (\Phil_B)_{ij} a_{j*} \\
\phi_B(a_{*i}) &= \sum_{i=1}^n a_{*j} (\Phir_B)_{ji}
\end{align*}
for $1\leq i\leq n$.
\end{definition}

\begin{exercise} \label{exc:PhiLPhiR}
\begin{enumerate}
\item
For $B = \sigma_1^3 \in B_3$, use arcs and Theorem~\ref{thm:arcs} to check that
\[
\phi_B(a_{13}) = -2a_{21}a_{13}+a_{21}a_{12}a_{21}a_{13}+a_{23}-a_{21}a_{12}a_{23}.
\]
\item
Now view $B = \sigma_1^3$ as living in $B_2$. Verify:
\begin{align*}
\Phil_B &= \left( \begin{matrix}
-2a_{21}+a_{21}a_{12}a_{21} & 1-a_{21}a_{12} \\
1-a_{12}a_{21} & a_{12} \end{matrix} \right) \\
\Phir_B &= \left( \begin{matrix}
-2a_{12}+a_{12}a_{21}a_{12} & 1-a_{12}a_{21} \\
1-a_{21}a_{12} & a_{21}
\end{matrix} \right).
\end{align*}
\item
For general $B$, $\Phil_B$ and $\Phir_B$ can be thought of as ``square
roots'' of $\phi_B$, in the following sense. Let $\AA$ and
$\phi_B(\AA)$ be the $n\times n$ matrices defined in
Definition~\ref{def:knotDGA} below; roughly speaking, $\AA$ is the
matrix of the $a_{ij}$'s and $\phi_B(\AA)$ is the matrix of the
$\phi_B(a_{ij})$'s. Then we have
\begin{equation}
\phi_B(\AA) = \Phil_B \cdot \AA \cdot \Phir_B; \label{eq:product}
\end{equation}
see \cite{Ngframed,transhom} for the proof. Verify \eqref{eq:product}
for $B = \sigma_1^3$.
\end{enumerate}
\end{exercise}

\begin{definition}[\cite{transhom,EENS}\footnote{See the Appendix for
    differences in convention between our definition and the ones from
    \cite{transhom} and \cite{EENS}.}]
Let $K$ be a knot given by the closure of a braid $B \in B_n$.
\label{def:knotDGA}
The \textit{(combinatorial) knot DGA} for $K$ is the differential graded algebra $(\A,\d)$ over $R = \Z[\lambda^{\pm 1},\mu^{\pm 1},U^{\pm 1}]$ given as follows.
\begin{enumerate}
\item
Generators: $\A = R\langle a_{ij},b_{ij},c_{ij},d_{ij},e_{ij},f_{ij} \rangle$ with generators
\begin{itemize}
\item
$a_{ij}$, where $1\leq i,j\leq n$ and $i\neq
j$, of degree $0$ ($n(n-1)$ of these)
\item
$b_{ij}$, where $1\leq i,j\leq n$ and $i\neq j$,
of degree $1$ ($n(n-1)$ of these)
\item
$c_{ij}$ and $d_{ij}$, where $1\leq i,j\leq n$, of
degree $1$ ($n^2$ of each)
\item
$e_{ij}$ and $f_{ij}$, where $1\leq i,j\leq n$, of
degree $2$ ($n^2$ of each).
\end{itemize}
\item
Differential: assemble the generators into $n\times n$ matrices
$\AA,\hat{\AA},\BB,\hat{\BB},\CC,\DD,\EE,\FF$, defined as follows. For
$1\leq i,j\leq n$, the $ij$ entry of the matrices $\CC,\DD,\EE,\FF$
is $c_{ij},d_{ij},e_{ij},f_{ij}$, respectively. The other matrices $\AA,\hat{\AA},\BB,\hat{\BB}$ are given by:
\begin{align*}
\AA_{ij} &= \begin{cases} a_{ij} & i<j \\
-\mu a_{ij} & i>j \\
1-\mu & i=j \\
\end{cases}
&
\BB_{ij} &= \begin{cases} b_{ij} & i<j \\
-\mu b_{ij} & i>j \\
0 & i=j \\
\end{cases}  \\
(\Ahat)_{ij} &= \begin{cases} U a_{ij} & i<j \\
-\mu a_{ij} & i>j \\
U-\mu & i=j \\
\end{cases}
&
(\Bhat)_{ij} &= \begin{cases} U b_{ij} & i<j \\
-\mu b_{ij} & i>j \\
0 & i=j. \\
\end{cases}
\end{align*}
Also define a matrix $\LL$ as the diagonal matrix
$\LL=\operatorname{diag}(\lambda\mu^{w}U^{-(w-n+1)/2},1,\ldots,1)$, where $w$ is the writhe of $B$ (the sum of the exponents in the braid word).

The differential is given in matrix form by:
\begin{align*}
\d(\AA) &= 0 \\
\d(\BB) &= \AA - \LL \cdot \phi_B(\AA) \cdot \LL^{-1} \\
\d(\CC) &= \Ahat - \LL \cdot \Phil_B \cdot \AA \\
\d(\DD) &= \AA - \Ahat \cdot \Phir_B \cdot \LL^{-1} \\
\d(\EE) &= \hat{\BB} - \CC - \LL \cdot \Phil_B \cdot \DD \\
\d(\FF) &= \BB - \DD - \CC \cdot \Phir_B \cdot \LL^{-1}.
\end{align*}
Here $\d(\AA)$ is the matrix whose $ij$ entry is $\d(\AA_{ij})$,
$\phi_B(\AA)$ is the matrix whose $ij$ entry is $\phi_B(\AA_{ij})$,
and similarly for $\d(\BB)$, etc.
(For $U=1$ as in the setting of \cite{Ngframed}, we can omit the hats.)

The homology of $(\A,\d)$ is the (combinatorial) knot contact homology
$HC_*(K)$.
\end{enumerate}

\end{definition}

\begin{remark}
Combinatorial knot DGAs and related invariants are readily calculable
by computer. There are a number of \textit{Mathematica} packages to
this end available at

\centerline{
\url{http://www.math.duke.edu/~ng/math/programs.html}
}
\end{remark}

\begin{example}
For the unknot, the knot DGA is the algebra over $\Z[\lambda^{\pm 1},\mu^{\pm 1},U^{\pm 1}]$ generated by four generators, $c,d$ in degree $1$ and $e,f$ in degree $2$, with differential:
\begin{align*}
\d c &= U-\lambda-\mu+\lambda\mu \\
\d d &= 1-\mu-\lambda^{-1}U+\lambda^{-1}\mu \\
\d e &= -c-\lambda d \\
\d f &= -d-\lambda^{-1} c.
\end{align*}
Up to stable tame isomorphism, this is the same as the DGA generated by $c$ and $e$ with differential
$\d c = U-\lambda-\mu+\lambda\mu$, $\d e=0$. \label{ex:unknot}
See Exercise~\ref{exc:sti}(\ref{item:unknotsti}).
\end{example}

The main result of \cite{EENS} is that the combinatorial knot DGA of
$K$, described above,
agrees with the LCH DGA of $\Lambda_K$, after one changes $\Lambda_K$
by Legendrian isotopy in $J^1(S^2)$ in a particular way and makes
other choices that do not affect LCH. The proof of this result is far
outside the scope of this article, but we will try to indicate the
strategy; see also \cite{EE} for a nice summary with a bit more detail.

\begin{theorem}[\cite{EENS,EENStransverse}]
\label{thm:EENS}
The combinatorial knot DGA of $K$ in the sense of
Definition~\ref{def:knotDGA} is
the LCH DGA of $\Lambda_K$ in the sense of Definition~\ref{def:KCH}.
\end{theorem}

\begin{proof}[Idea of proof]
Braid $K$ around an unknot $U$. Then $\Lambda_K$ is contained in a neighborhood of $\Lambda_U \cong T^2$, and so we can view
\[
\Lambda_K \subset J^1(T^2) \subset J^1(S^2)
\]
by the Legendrian neighborhood theorem. Reeb chords for $\Lambda_K$
split into two categories: ``small'' chords lying in $J^1(T^2)$,
corresponding to the $a_{ij}$'s and $b_{ij}'s$, and ``big'' chords
that lie outside of $J^1(T^2)$, corresponding to the
$c_{ij},d_{ij},e_{ij},f_{ij}$ generators (which themselves correspond
to four Reeb chords for $\Lambda_U$). Holomorphic disks similarly
split into small disks lying in $J^1(T^2)$, and big disks that lie
outside of $J^1(T^2)$. The small disks produce the subalgebra of the
knot DGA generated by the $a_{ij}$'s and $b_{ij}$'s. The big disks
produce the rest of the differential, and can be computed in the limit
degeneration when $K$ approaches $U$. These disk counts use gradient
flow trees in the manner of \cite{EkholmFlowTrees}.
\end{proof}

It follows from Theorem~\ref{thm:EENS} that the combinatorial knot
DGA, up to stable tame isomorphism,
is a knot invariant, as is its homology $HC_*(K)$. Alternatively, one
can prove this directly without counting holomorphic curves, just by
using algebraic properties of the representation $\phi$ and the
matrices $\Phil_B,\Phir_B$.

\begin{theorem}[\cite{Ngframed} for $U=1$, \cite{transhom} in general]
For the combinatorial knot DGA: \label{thm:invariance}
\begin{enumerate}
\item
$\d^2=0$ (see Exercise~\ref{exc:d2});
\item
$(\A,\d)$ is a knot invariant: up to stable tame isomorphism, it is invariant under Markov moves.
\end{enumerate}
\end{theorem}

\begin{exercise}
\label{exc:d2}
\begin{enumerate}
\item
Use \eqref{eq:product} from Exercise~\ref{exc:PhiLPhiR} to prove that
$\d^2=0$ for the combinatorial knot DGA.
\item
Show that the two-sided ideal in $\A$ generated by the entries of any
two of the three matrices $\AA - \LL \cdot \phi_B(\AA) \cdot
\LL^{-1}$, $\Ahat - \LL \cdot \Phil_B \cdot \AA$,
$\AA - \Ahat \cdot \Phir_B \cdot \LL^{-1})$ contains
the entries of the third.
\label{item:combDGA}
(Note that these three matrices are the
matrices of differentials $\d(\BB)$, $\d(\CC)$, $\d(\DD)$ in the knot DGA.)
This fact will appear later; see Remark~\ref{rmk:two}.
\end{enumerate}
\end{exercise}

It is natural to ask how effective the knot DGA is as a knot
invariant. In order to answer this, one needs to find practical ways
of distinguishing between stable tame isomorphism classes of DGAs. One
way, outlined in the following exercise, is by linearizing, as in
Exercise~\ref{exc:lin}; another, which we will employ and discuss
extensively later, is by considering the space of augmentations, as in
Remark~\ref{rmk:augvar}.

\begin{exercise} \label{exc:canaug}
\begin{enumerate}
\item
Show that the knot DGA has an augmentation to $\Z[\lambda^{\pm 1}]$
that sends $\mu,U$ to $1$, and another augmentation to $\Z[\mu^{\pm
  1}]$ that sends $\lambda,U$ to $1$. (In general there are many more
augmentations, but these are ``canonical'' in some sense.) Hint: this is easiest to do using
the cord algebra (see Section~\ref{sec:cordalg}) rather than the knot
DGA directly.
\item
Consider the right-handed trefoil $K$, expressed as the closure of
$\sigma_1^3 \in B_2$. If we further compose the second augmentation
from the previous part with the map $\Z[\mu^{\pm 1}] \to \Z$ that
sends $\mu$ to $-1$, then we obtain an augmentation of the knot DGA of
$K$ to $\Z$. This is explicitly given as the map $\epsilon : \A \to
\Z$ with $\epsilon(\lambda)=1$, $\epsilon(\mu)=-1$, $\epsilon(U)=1$,
$\epsilon(a_{12}) = \epsilon(a_{21}) = -2$.

For this augmentation, show that the linearized contact homology (see
Exercise~\ref{exc:lin})
$HC^{\text{lin}}_*(\Lambda_K,\epsilon)$ is given as follows:
\[
HC^{\text{lin}}_* \cong \begin{cases}
\Z_3 & *=0 \\
\Z \oplus (\Z_3)^3 & *=1 \\
\Z & *=2 \\
0 & \text{otherwise}.
\end{cases}
\]
\item
By contrast, check that for the unknot (whose DGA is given at the end of
Example~\ref{ex:unknot}), there is a unique augmentation to $\Z$ with
$\epsilon(\lambda)=1$, $\epsilon(\mu)=-1$, $\epsilon(U)=1$, with
respect to which $HC^{\text{lin}}_0 \cong 0$, $HC^{\text{lin}}_1 \cong
\Z$, $HC^{\text{lin}}_2 \cong \Z$. It can be shown (see \cite{Ch02})
that the collection of all linearized homologies over all possible
augmentations is an invariant of the stable tame isomorphism class of
a DGA. Thus the knot DGAs for the unknot and right-handed trefoil are
not stable tame isomorphic.
\end{enumerate}
\end{exercise}

We close this section by discussiong some properties of the knot DGA,
which are proved using the combinatorial formulation from
Definition~\ref{def:knotDGA}.

\begin{theorem}[\cite{Ngframed}]
\label{thm:knotDGAproperties}
\begin{enumerate}
\item
Knot contact homology encodes the Alexander polynomial: there is a
canonical augmentation of the knot DGA $(\A,\d)$ to $\Z[\mu^{\pm 1}]$
(see Exercise~\ref{exc:canaug}), with respect to which the linearized
contact homology $HC_*^{\text{lin}}(K)$, as a module over $\Z[\mu^{\pm
  1}]$, is such that $HC_1^{\text{lin}}(K)$ determines the
Alexander module of $K$ (see \cite{Ngframed} for the precise
statement).
\item
Knot contact homology detects mirrors and mutants: counting
augmentations to $\Z_3$ shows that the knot DGAs for the right-handed
and left-handed trefoils and the Kinoshita--Terasaka and Conway
mutants are all distinct.
\end{enumerate}
\end{theorem}

\begin{remark}
Since the knot DGA $(\A,\d)$ is supported in nonnegative degree,
augmentations to $\Z_3$ (or arbitrary rings) are the same as ring
homomorphisms from $HC_0(K)$ to $\Z_3$; see
Exercise~\ref{exc:sti}. Thus the number of such augmentations is a
knot invariant. Counting augmentations to finite fields is easy to do
by computer.
\end{remark}

\begin{remark}
It is not known if there are nonisotopic knots $K_1,K_2$ whose knot
contact homologies are the same. Thus at present it is conceivable
that any of the following are \textit{complete} knot invariants, in
decreasing order of strength of the invariant (except possibly for the last two
items, which do not determine each other in any obvious way):
\begin{itemize}
\item
the Legendrian isotopy class of $\Lambda_K \subset ST^*\R^3$;
\item
the knot DGA $(\A,\d)$ up to stable tame isomorphism;
\item
degree $0$ knot contact homology $HC_0(K)$ over $R = \Z[\lambda^{\pm 1},\mu^{\pm 1},U^{\pm 1}]$;
\item
the cord algebra (see Section~\ref{sec:cordalg});
\item
the augmentation polynomial $\Aug_K(\lambda,\mu,U)$ (see Section~\ref{sec:aug}).
\end{itemize}
Even if these are not complete invariants, they are rather strong. For
instance, physics arguments suggest that the augmentation polynomial
may be at least as strong as the HOMFLY-PT polynomial and possibly some
knot homologies; see Section~\ref{sec:aug}.
\end{remark}

\section{Cord Algebra}
\label{sec:cordalg}

In the previous section, we introduced the (combinatorial) knot
DGA. The fact that the knot DGA is a topological invariant can be
shown in two ways: computation of holomorphic disks and an appeal to the
general theory of Legendrian contact homology as in
Section~\ref{sec:LCH} \cite{EENS}, or combinatorial verification of invariance
under the Markov moves \cite{transhom}. The first approach is natural
but difficult, while the second is technically easier but somewhat
opaque from a topological viewpoint, a bit like the usual proofs
that the Jones polynomial is a knot invariant.

In this section, we
present a direct topological interpretation for a significant part
(though not the entirety) of knot contact homology, namely the degree $0$
homology $HC_0(K)$ with $U=1$, in terms of a construction called the
``cord algebra''. Our aim is to give some topological
intuition for what knot contact homology measures as a knot
invariant. It is currently an open problem to extend this interpretation to all
of knot contact homology.

We begin with the observation that $HC_*(K)$ is supported in degree $*
\geq 0$, and that for
$*=0$ it can be written fairly explicitly:

\begin{theorem}
Let $R = \Z[\lambda^{\pm 1},\mu^{\pm 1},U^{\pm 1}]$. Then \label{thm:HC0matrix}
\[
HC_0(K) \cong (\A_n \otimes R) \, / \, (\text{entries of }
\AA - \LL \cdot \phi_B(\AA) \cdot \LL^{-1},~
\Ahat - \LL \cdot \Phil_B \cdot \AA,~
\AA - \Ahat \cdot \Phir_B \cdot \LL^{-1}).
\]
\end{theorem}

\begin{proof}
Since the knot DGA $(\A,\d)$ is supported in degree $\geq 0$, all
degree $0$ elements of $\A$, i.e., elements of $\A_n \otimes R$, are
cycles. The ideal of $\A_n \otimes R$ consisting of boundaries is
precisely the ideal generated by the entries of the three matrices.
\end{proof}

\begin{remark}
In fact, one can drop any single one of the matrices $\AA -
\linebreak \LL \cdot \phi_B(\AA) \cdot \LL^{-1}$,
$\Ahat - \LL \cdot \Phil_B \cdot \AA$,
$\AA - \Ahat \cdot \Phir_B \cdot \LL^{-1}$ in the statement of
Theorem~\ref{thm:HC0matrix}.
\label{rmk:two}
See Exercise~\ref{exc:d2}(\ref{item:combDGA}).
\end{remark}

\begin{remark}
It does not appear to be an easy task to find an analogue of
Theorem~\ref{thm:HC0matrix} for $HC_*(K)$ with $* \geq 1$, in part because not
all elements of $\A$ of the appropriate degree are cycles.
\end{remark}

Although the expression for $HC_0(K)$ from Theorem~\ref{thm:HC0matrix}
is computable in examples, it has a particularly nice interpretation
if we set $U=1$, as we will do for the rest of this section. With
$U=1$, the coefficient ring for the knot DGA becomes $R_0 =
\Z[\lambda^{\pm 1},\mu^{\pm 1}]$, and we can express $HC_0(K)|_{U=1}$
as an algebra over $R_0$ generated by ``cords''.

\begin{definition}[\cite{NgKCH2,Ngframed}] \label{def:cordalg}
\begin{enumerate}
\item
Let $(K,*) \subset S^3$ be an oriented knot with a basepoint. A \textit{cord} of $(K,*)$ is a continuous path $\gamma :\thinspace [0,1] \to S^3$ with $\gamma^{-1}(K) = \{0,1\}$ and $\gamma^{-1}(\{*\})=\emptyset$.
\item
Define $\A_K$ to be the tensor algebra over $R_0$ freely generated by homotopy classes of cords (note: the endpoints of the cord can move along the knot, as long as they avoid the basepoint $*$).
\item
The \textit{cord algebra} of $K$ is the algebra $\A_K$ modulo the relations:
\begin{enumerate}
\item
\label{eq:skein1new}
$\raisebox{-0.17in}{\includegraphics[width=0.4in]{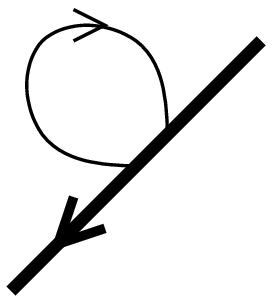}}
= 1-\mu$
\item
\label{eq:skein2new}
$\raisebox{-0.17in}{\includegraphics[width=0.4in]{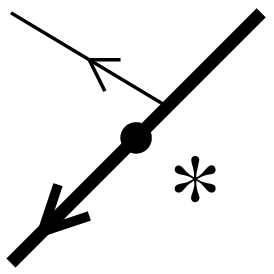}}
= \lambda
\raisebox{-0.17in}{\includegraphics[width=0.4in]{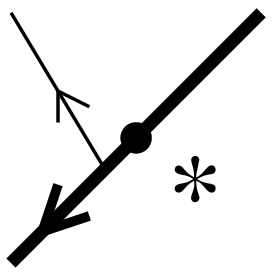}}$
\textrm{and}
$\raisebox{-0.17in}{\includegraphics[width=0.4in]{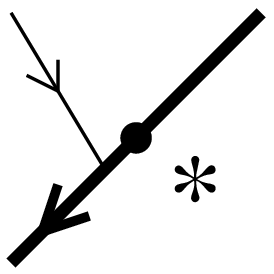}}
= \lambda
\raisebox{-0.17in}{\includegraphics[width=0.4in]{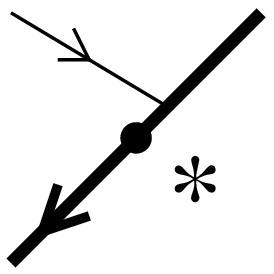}}$
\item
\label{eq:skein3new}
$\raisebox{-0.17in}{\includegraphics[width=0.4in]{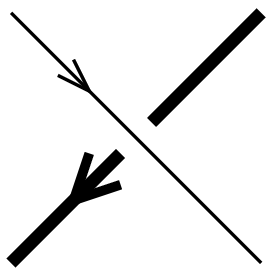}}
-
\mu \raisebox{-0.17in}{\includegraphics[width=0.4in]{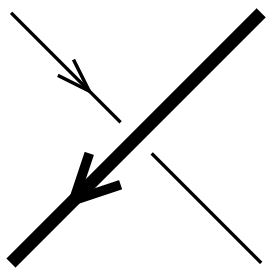}}
-
\raisebox{-0.17in}{\includegraphics[width=0.4in]{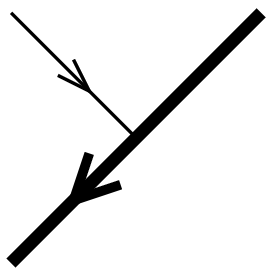}}
\cdot
\raisebox{-0.17in}{\includegraphics[width=0.4in]{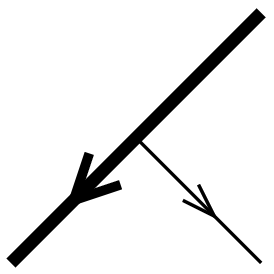}} = 0$.
\end{enumerate}
\end{enumerate}
\end{definition}

The ``skein relations'' in Definition~\ref{def:cordalg} are understood
to be depictions of relations in $\R^3$, and not just relations as
planar diagrams. For instance, relation \eqref{eq:skein3new} is
equivalent to:
\[
\raisebox{-0.17in}{\includegraphics[width=0.4in]{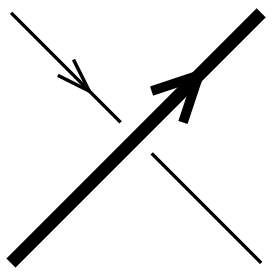}}
-\mu
\raisebox{-0.17in}{\includegraphics[width=0.4in]{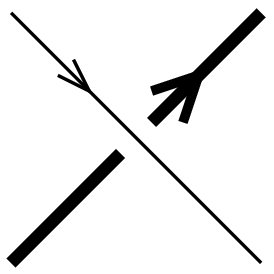}}
-\raisebox{-0.17in}{\includegraphics[width=0.4in]{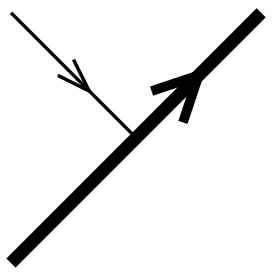}}
\cdot
\raisebox{-0.17in}{\includegraphics[width=0.4in]{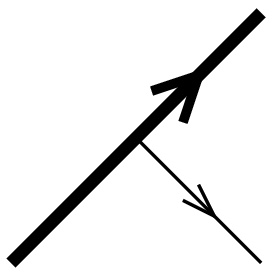}} =
0.
\]
It is then evident that the cord algebra is a topological knot invariant.

\begin{exercise}
One can heuristically think of cords as corresponding to Reeb chords
of $\Lambda_K$. More precisely: \label{exc:Reeb}
\begin{enumerate}
\item
Let $K \subset \R^3$ be a smooth knot.
A \textit{binormal chord} of $K$ is an oriented (nontrivial) line
segment with endpoints on $K$ that is orthogonal to $K$ at both
endpoints. Show that binormal chords are exactly the same as Reeb
chords of $\Lambda_K$.
\item
For generic $K$, all binormal chords are cords in the sense of
Definition~\ref{def:cordalg}. Show that any element of the cord
algebra of $K$ can be
expressed in terms of just binormal chords, i.e., in terms of Reeb
chords of $\Lambda_K$.
\item
Prove that the cord algebra of a $m$-bridge knot has a presentation
with (at most) $m(m-1)$ generators. (It is currently unknown whether
this also holds for $HC_0$ if we do not set $U=1$.)
\item
Prove that the cord algebra of the torus knot $T(m,n)$ has a
presentation with at most $\min(m,n)-1$ generators, as indeed does
$HC_0(T(m,n))$ without setting $U=1$. (For this last statement, see Exercise~\ref{exc:homom}(\ref{item:twist}).)
\end{enumerate}
\end{exercise}

\begin{exercise}
Here we calculate the cord algebra in two simple examples.
\label{exc:trefcordalg}
\begin{enumerate}
\item
Prove that the cord algebra of the unknot is $R_0/((\lambda-1)(\mu-1))$.
\item
Next consider the right-handed trefoil $K$, shown below with five
cords labeled:

\centerline{
\includegraphics[height=2in]{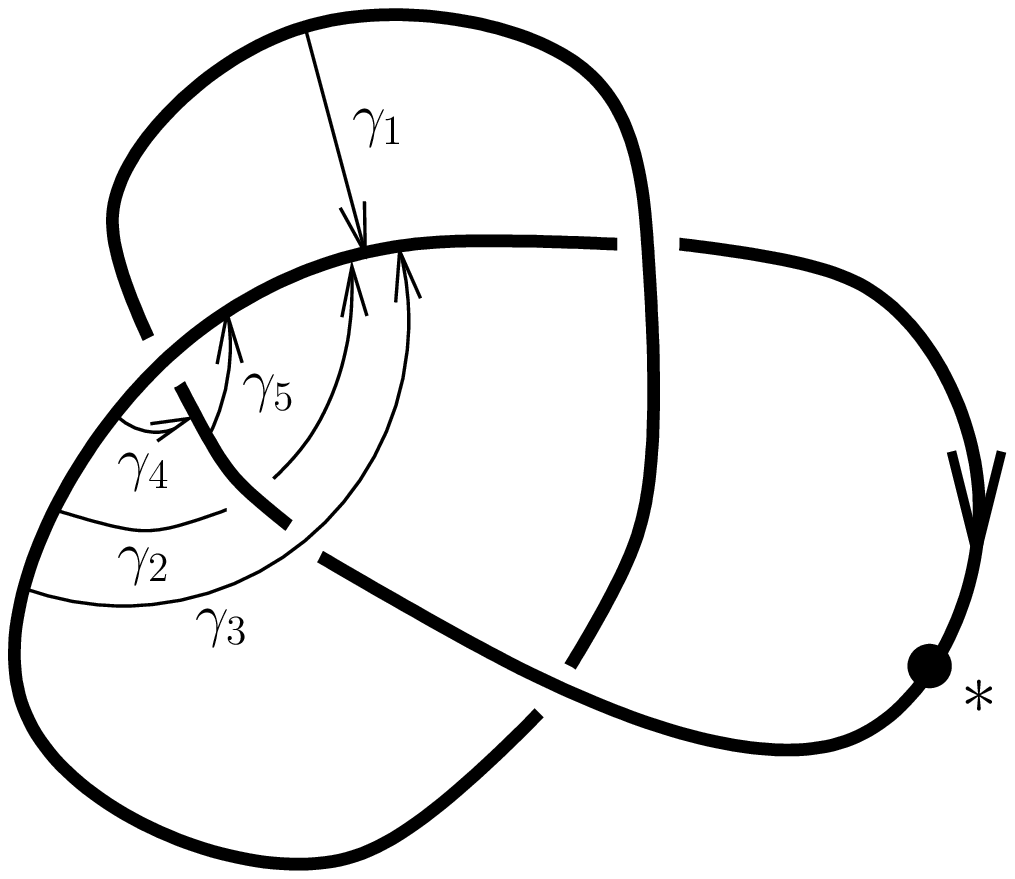}
}

\noindent
In the cord algebra of $K$, denote $\gamma_1$ by $x$. Show that
$\gamma_2=\gamma_5=x$, $\gamma_4=\lambda x$, and $\gamma_3=1-\mu$. Conclude the relation
\[
\lambda x^2-x+\mu-\mu^2 = 0.
\]
\item
Use the skein relations in another way to derive another relation in
the cord algebra of $K$:
\[
\lambda x^2+\lambda\mu x+\mu-1 = 0.
\]
\item
Prove that the cord algebra of $K$ is generated by $x$.
\item
It can be shown that the above two relations generate all relations: the
cord algebra of the right-handed trefoil is
\[
R_0[x] \,/\, \left(
\lambda x^2-x+\mu-\mu^2,~
\lambda x^2+\lambda\mu x+\mu-1 \right).
\]
Suppose that there is a ring homomorphism from the cord algebra of $K$
to $\C$, mapping $\lambda$ to $\lambda_0$ and $\mu$ to $\mu_0$. Show that
\[
(\lambda_0-1)(\mu_0-1)(\lambda_0 \mu_0^3+1)=0.
\]
The left hand side is the two-variable augmentation polynomial for the
right-handed trefoil (see Section~\ref{sec:aug} and Example~\ref{ex:2varaug}).
\end{enumerate}
\end{exercise}

We now present the relation between the cord algebra and knot contact homology.

\begin{theorem}[\cite{NgKCH2,Ngframed}]
The cord algebra of $K$ is isomorphic as an $R_0$-algebra
\label{thm:kchcord} to $HC_0(K)|_{U=1}$.
\end{theorem}

\begin{proof}[Idea of proof]
Let $K$ be the closure of a braid $B \in B_n$, and embed $B$ in $S^3$ with
braid axis $L$. A page of the resulting open book decomposition of
$S^3$ is $D^2$ with $\d D^2 = L$, and $D^2$ intersects $B$ in $n$
points $p_1,\ldots,p_n$. Any arc in $D^2 \subset S^3$ in the sense of
Definition~\ref{def:arc} is a cord of $K$. Under this identification,
skein relations \eqref{eq:skein3new} and \eqref{eq:skein1new} from
Definition~\ref{def:cordalg} become relations \ref{item:arc1} and
\ref{item:arc2} from Definition~\ref{def:arc} (at least when $\mu=1$;
for general $\mu$, one needs to use a variant of
Definition~\ref{def:arc} involving framed cords, cf.\
Remark~\ref{rmk:framedarc}).

Any cord of $K$ is homotopic to a cord lying in the $D^2$ slice of
$S^3$. It then follows from Theorem~\ref{thm:arcs}
that there is a surjective $R_0$-algebra map from $\A_n \otimes R_0$
to the cord algebra. Thus the cord algebra is the quotient of $\A_n
\otimes R_0$ by relations that arise from
considering  homotopies between arcs in $D^2$ given by one-parameter families of
cords that do not lie in the $D^2$ slice. If this family avoids
intersecting $L$, we obtain the relations given by the entries of
$\d(\BB) = \AA - \LL \cdot \phi_B(\AA) \cdot \LL^{-1}$. Considering
families that pass through $L$ once gives the entries of
$\d(\CC) = \Ahat - \LL \cdot \Phil_B \cdot \AA$ and
$\d(\DD) = \AA - \Ahat \cdot \Phir_B \cdot \LL^{-1}$ as relations in
the cord algebra.
\end{proof}

For various purposes, it is useful to reformulate the cord algebra of
a knot $K$ in terms of homotopy-group information. In particular, this
gives a proof that knot contact homology detects the unknot
(Corollary~\ref{cor:unknot}); in Section~\ref{sec:aug}, we will
also use this to relate the augmentation polynomial to the $A$-polynomial.
Here we give a
brief description of this perspective and refer the reader to
\cite{Ngframed}
for more details.

We can view cords of $K$ as elements of the knot group
$\pi_1(S^3\setminus K)$ by pushing the endpoints slightly off of $K$ and
joining them via a curve parallel to $K$. One can then present the
cord algebra entirely in terms of the knot group
$\pi$ and the peripheral subgroup $\hat{\pi} =
\pi_1(\d(\text{nbd}(K))) \cong \Z^2$. Write $l,m$ for the longitude,
meridian generators of $\hat{\pi}$.

\begin{theorem}[\cite{Ngframed}]
The cord algebra of $K$ is isomorphic to the tensor algebra over $R_0$ freely generated by elements of $\pi_1(S^3\setminus K)$ (denoted with brackets), quotiented by the relations: \label{thm:htpy}
\begin{enumerate}
\item
$[e] = 1-\mu$, where $e$ is the identity element;
\item
$[\gamma l] = [l\gamma] = \lambda [\gamma]$ and
\label{item:htpy2}
$[\gamma m]=[m \gamma]
= \mu[\gamma]$ for $\gamma\in \pi_1(S^3\setminus K)$;
\item
$[\gamma_1\gamma_2] - [\gamma_1 m \gamma_2] - [\gamma_1] [\gamma_2] =
0$ for any $\gamma_1,\gamma_2\in \pi_1(S^3\setminus K)$.
\end{enumerate}
\end{theorem}

If $(\A,\d)$ is the knot DGA of $K$, then Theorem~\ref{thm:htpy}
(along with Theorem~\ref{thm:kchcord}) gives
an expression for $HC_0(K)|_{U=1} = H_0(\A|_{U=1},\d)$ as an
$R_0$-algebra.
One can readily ``improve'' this result to give an analogous
expression for the degree $0$ homology of the fully noncommutative
knot DGA $(\widetilde{\A},\d)$ of $K$
(see Remark~\ref{rmk:noncommDGA} and the Appendix), which we write as
\[
\widetilde{HC}_0(K)|_{U=1} = H_0(\widetilde{\A}|_{U=1},\d);
\]
note that this is a $\Z$-algebra rather than a $R_0$-algebra, but
contains $R_0$ as a subalgebra.
Details are contained in joint work in progress with K.\ Cieliebak,
T.\ Ekholm, and J.\ Latschev, which is also the reference for
Theorem~\ref{thm:celn} and Corollary~\ref{cor:unknot} below.

\begin{theorem}
Write $\pi = \pi_1(S^3\setminus K)$ and $\hat{\pi} =
\pi_1(\d(\text{nbd}(K))) = \langle m,l\rangle$.
There is an injective ring homomorphism \label{thm:celn}
\[
\widetilde{HC}_0(K)|_{U=1} \hookrightarrow \Z[\pi_1(S^3\setminus K)]
\]
under which $\widetilde{HC}_0(K)|_{U=1}$ maps isomorphically to the subring of $\Z[\pi]$ generated by $\hat{\pi}$ and elements of the form $\gamma-m\gamma$ for $\gamma\in\pi$. This map sends $\lambda$ to $l$ and $\mu$ to $m$.
\end{theorem}

\begin{proof}[Idea of proof]
The homomorphism is induced by the map sending $\lambda$ to $l$, $\mu$
to $m$, and $[\gamma]$
to $\gamma-m \gamma$ for $\gamma\in\pi$.
\end{proof}

\begin{corollary}
Knot contact homology, in its fully noncommutative form, detects the
unknot. \label{cor:unknot}
\end{corollary}

\begin{proof}[Idea of proof]
Use the Loop Theorem and consider the action of multiplication by $\lambda$ on the cord algebra.
\end{proof}

For a proof that ordinary (not fully noncommutative) knot contact
homology detects the unknot, see the next section.

\section{Augmentation Polynomial}
\label{sec:aug}

In this section, we describe how knot contact homology can be used to
produce a three-variable knot invariant, the augmentation
polynomial. We then discuss the relation of a two-variable version of
the augmentation polynomial to the $A$-polynomial, and of the full
augmentation polynomial to the HOMFLY-PT polynomial and to mirror
symmetry and physics.

The starting point is the space of augmentations from the knot DGA
$(\A,\d)$ to $\C$, as in Remark~\ref{rmk:augvar}.

\begin{definition}[\cite{Ngframed,transhom}]
Let $(\A,\d)$ be the knot DGA of a knot $K$, with the usual
coefficient ring $\Z[\lambda^{\pm
  1},\mu^{\pm 1},U^{\pm 1}]$.
\label{def:augpoly}
The \textit{augmentation variety} of $K$
is
\[
V_K = \{(\epsilon(\lambda),\epsilon(\mu),\epsilon(U)) \,|\,
\epsilon \text{ an augmentation from } (\A,\d) \text{ to } \C \} \subset (\C^*)^3.
\]
When the maximal-dimension part of the Zariski closure of $V_K$ is a
codimension $1$ subvariety of $(\C^*)^3$, this variety is the
vanishing set of a reduced polynomial\footnote{I.e., no repeated
  factors.} $\Aug_K(\lambda,\mu,U)$, the
\textit{augmentation polynomial}\footnote{Caution: the polynomial
  described here differs from the augmentation polynomial from
  \cite{transhom} by a change of variables $\mu \mapsto -1/\mu$. See
  the Appendix.}
 of $K$.
\end{definition}

\begin{remark}
The augmentation polynomial is well-defined only up to units in
$\C[\lambda^{\pm 1},\mu^{\pm  1},U^{\pm 1}]$.
\label{rmk:Zpoly}
However, because the
differential on the knot DGA involves only integer coefficients, we
can choose $\Aug_K(\lambda,\mu,U)$ to have integer coefficients with
overall $\gcd$ equal to $1$. We can further stipulate that
$\Aug_K(\lambda,\mu,U)$ contains no negative powers of
$\lambda,\mu,U$, and that it is divisible by none of
$\lambda,\mu,U$. The result is an augmentation polynomial
$\Aug_K(\lambda,\mu,U) \in \Z[\lambda,\mu,U]$, well-defined up to an
overall $\pm$ sign.
\end{remark}

\begin{conjecture}
The condition about the Zariski closure in
Definition~\ref{def:augpoly} holds for all knots $K$; the augmentation
polynomial is always defined.
\end{conjecture}

A fair number of augmentation polynomials for knots have been been
computed and are available at

\centerline{
\url{http://www.math.duke.edu/~ng/math/programs.html};
}

\noindent see also Exercise~\ref{exc:augpolycomputations} below.
We note in passing some symmetries of the augmentation polynomial:

\begin{theorem}
Let $K$ be a knot and $m(K)$ its mirror. Then
\label{thm:symmetries}
\[
\Aug_K(\lambda,\mu,U) \doteq \Aug_K(\lambda^{-1} U,\mu^{-1} U,U)
\]
and
\[
\Aug_{m(K)}(\lambda,\mu,U) \doteq
\Aug_{K}(\lambda U^{-1},\mu^{-1},U^{-1}),
\]
where $\doteq$ denotes equality up to units in $\Z[\lambda^{\pm
  1},\mu^{\pm 1},U^{\pm 1}]$.
\end{theorem}

\noindent
The first equation in Theorem~\ref{thm:symmetries} follows from
\cite[Propositions~4.2,4.3]{transhom}, while the second can be proved using
the results from \cite[\S 4]{transhom}.

\begin{exercise}
Here are a couple of computations
\label{exc:augpolycomputations}
of augmentation polynomials.
\begin{enumerate}
\item
Show that the augmentation polynomial for the unknot is
\[
\Aug_O(\lambda,\mu,U) =
U-\lambda-\mu+\lambda\mu.
\]
\item
The cord algebra $HC_0|_{U=1}$ for the right-handed trefoil was computed in
Exercise~\ref{exc:trefcordalg}. It can be checked directly from the
definition of the knot DGA that the full degree $0$ knot contact
homology is
\[
HC_0(\text{RH trefoil}) \cong
R[a_{12}]\,/\,( U a_{12}^2-\mu U a_{12}+\lambda\mu^3(1-\mu),
U a_{12}^2+\lambda\mu^2 a_{12}+\lambda\mu^2(\mu-U)).
\]
Use resultants to deduce the augmentation polynomial:
\begin{align*}
\Aug_{\text{RH trefoil}}(\lambda,\mu,U) &=
(U^3-\mu U^2)+(-U^3+\mu U^2-2\mu^2U+2\mu^2U^2 \\
& \qquad +\mu^3U-\mu^4U)\lambda+(-\mu^3+\mu^4)\lambda^2.
\end{align*}
From Theorem~\ref{thm:symmetries}, we can then also deduce the
polynomial for the left-handed trefoil:
\begin{align*}
\Aug_{\text{LH trefoil}}(\lambda,\mu,U) &=
(\mu^3 U^2-\mu^4 U)+(U^2-\mu U^2-2\mu^2 U+2\mu^2 U^2 \\
& \qquad -\mu^3
U+\mu^4)\lambda+ (-U^2+\mu U^2)\lambda^2.
\end{align*}
\end{enumerate}
\end{exercise}

We next turn to the two-variable augmentation polynomial.

\begin{definition}[\cite{Ngframed}]
If the $U=1$ slice of the augmentation variety, $V_K \cap \{U=1\}
\subset (\C^*)^2$, is
such that the maximal-dimensional part of its Zariski closure is
a (co)dimension $1$ subvariety of $(\C^*)^2$, then this subvariety is
the vanishing set of a reduced polynomial $\Aug_K(\lambda,\mu)$, the
\textit{two-variable augmentation polynomial} of $K$. As in
Remark~\ref{rmk:Zpoly}, $\Aug_K(\lambda,\mu)$ can be chosen to lie in
$\Z[\lambda,\mu]$.
\end{definition}

\begin{conjecture}
The two-variable augmentation polynomial $\Aug_K(\lambda,\mu)$ is
always defined, and the two augmentation polynomials are related in
the obvious way:
\[
\Aug_K(\lambda,\mu) = \Aug_K(\lambda,\mu,U=1).
\]
\end{conjecture}

The two-variable augmentation polynomial has a number of interesting
factors. For instance, it follows from Exercise~\ref{exc:canaug} that
\[
(\lambda-1)(\mu-1) \,|\, \Aug_K(\lambda,\mu)
\]
for all knots $K$.

\begin{example}
For the unknot and trefoils,
\label{ex:2varaug} the two-variable augmentation polynomials are
\begin{align*}
\Aug_O(\lambda,\mu) &= (\lambda-1)(\mu-1) \\
\Aug_{\text{RH trefoil}}(\lambda,\mu) &= (\lambda-1)(\mu-1)(\lambda\mu^3+1) \\
\Aug_{\text{LH trefoil}}(\lambda,\mu) &= (\lambda-1)(\mu-1)(\lambda+\mu^3). \\
\end{align*}
The polynomial for the right-handed trefoil follows from
Exercise~\ref{exc:trefcordalg}, while the polynomial for the left-handed
trefoil follows from the behavior of the polynomial (and knot contact
homology generally) under mirroring, cf.\ Theorem~\ref{thm:symmetries}.
\end{example}

The observant reader may notice that the two-variable augmentation
polynomials for the unknot and trefoils are essentially the same as another knot polynomial, the $A$-polynomial. Recall that the $A$-polynomial is
defined as follows. Given an $SL_2\C$ representation of the knot group
\[
\rho :\thinspace \pi_1(S^3\setminus K)\to SL_2\C,
\]
simultaneously diagonalize $\rho(l)$, $\rho(m)$ to get $\rho(l) =
\left( \begin{smallmatrix} \lambda & * \\ 0 &
    \lambda^{-1} \end{smallmatrix} \right)$, $\rho(m) =
\left( \begin{smallmatrix} \mu & * \\ 0 & \mu^{-1} \end{smallmatrix}
\right)$. The (maximal-dimensional part of the
Zariski closure of the) collection of
$(\lambda,\mu)$ over all $SL_2\C$ representations is the zero set of
the \textit{$A$-polynomial} of $K$, $A_K(\lambda,\mu)$.

\begin{theorem}[\cite{Ngframed}]
$(\mu^2-1) A_K(\lambda,\mu)$ divides $\Aug_K(\lambda,\mu^2).$ \label{thm:Apoly}
\end{theorem}

We outline the proof of Theorem~\ref{thm:Apoly} in
Exercise~\ref{exc:Apoly} below.

\begin{corollary}
The cord algebra detects the unknot. \label{cor:unknot2}
\end{corollary}

\begin{proof}
By a result of Dunfield and Garoufalidis \cite{DG}, based on gauge-theoretic work of Kronheimer and Mrowka \cite{KMSU2}, the $A$-polynomial
detects the
unknot. It follows that when $K$ is knotted, either
$\Aug_K(\lambda,\mu)$ is not defined (if the augmentation variety is
$2$-dimensional), or $\Aug_K(\lambda,\mu^2)$
has a factor besides $(\lambda-1)(\mu-1)$. In either case, the
augmentation variety for $K$ is distinct from the variety for the
unknot, which is $\{\lambda=1\} \cup \{\mu=1\}$ (see Example~\ref{ex:2varaug}).
\end{proof}

Note that the statement of unknot detection in
Corollary~\ref{cor:unknot2} differs from, and is slightly stronger than, the statement
from Corollary~\ref{cor:unknot}, because of the issue of
commutativity. However, the proof of Corollary~\ref{cor:unknot}
uses only the Loop Theorem, rather than the deep Kronheimer--Mrowka
result that leads to Corollary~\ref{cor:unknot2}.

To expand on Theorem~\ref{thm:Apoly}, it is sometimes, but not always,
the case that
\[
\Aug_K(\lambda,\mu^2) = (\mu^2-1) A_K(\lambda,\mu).
\]
In general, the left hand side can contain factors that do not appear
in the right hand side. For example,
\begin{align*}
A_{T(3,4)}(\lambda,\mu) &= (\lambda-1)(\lambda\mu^{12}+1)(\lambda\mu^{12}-1) \\
\Aug_{T(3,4)}(\lambda,\mu) &= (\lambda-1)(\mu-1)(\lambda\mu^{6}+1)(\lambda\mu^{6}-1)(\lambda\mu^{8}-1),
\end{align*}
and the last factor in $\Aug_{T(3,4)}$ has no corresponding factor in
$A_{T(3,4)}$.

An explanation for (at least some of the) extra factors
in the augmentation polynomial is given by the following result, which
shows that representations of the knot group besides $SU_2$
representations can contribute to the augmentation polynomial.

\begin{theorem}
Suppose that $\rho :\thinspace \pi_1(S^3 \setminus K) \to GL_m\C$
\label{thm:rep}
is a
representation of the knot group of $K$ for some $m \geq 2$, such that
$\rho$ sends the meridian and longitude to the diagonal matrices
\begin{align*}
\rho(m) &= \diag(\mu_0,1,1,\ldots,1) \\
\rho(l) &= \diag(\lambda_0,\ast,\ast,\ldots,\ast)
\end{align*}
where the asterisks indicate arbitrary complex numbers. Then there is
an augmentation of the knot DGA of $K$ sending
$(\lambda,\mu,U)$ to $(\lambda_0,\mu_0,1)$.
\end{theorem}

\noindent
This result, which has not previously appeared in the literature, is
proven in the following exercise, and also implies Theorem~\ref{thm:Apoly}.

\begin{exercise}
Here we give a proof of Theorems~\ref{thm:Apoly} and~\ref{thm:rep}.
\label{exc:Apoly}
\begin{enumerate}
\item
Suppose $\rho :\thinspace \pi_1(S^3 \setminus K) \to GL_m\C$ is a
representation as in Theorem~\ref{thm:rep}. Define a $\C$-valued map
$\epsilon$ by
\begin{itemize}
\item
$\epsilon(\mu) = \mu_0$;
\item
$\epsilon(\lambda) = \lambda_0$;
\item
$\epsilon([\gamma]) = (1-\mu_0) \left( \rho(\gamma)
\right)_{11}$, where $M_{11}$ is the $(1,1)$ entry of a matrix $M$,
for all $\gamma\in\pi_1(S^3\setminus K)$.
\end{itemize}
Show that $\epsilon$ extends to an augmentation of the cord algebra of $K$,
where we use the description of the cord algebra from
Theorem~\ref{thm:htpy}. Deduce Theorem~\ref{thm:rep}.
\item
If $\rho$ is an $SU_2$ representation of $\pi_1(S^3 \setminus K)$ with
$\rho(m) = \left( \begin{smallmatrix} \mu & 0 \\ 0 &
    \mu^{-1} \end{smallmatrix} \right)$ and $\rho(l) =
\left( \begin{smallmatrix} \lambda & 0 \\ 0 &
    \lambda^{-1} \end{smallmatrix} \right)$, then show that
\[
\widetilde{\rho}(\gamma) = \mu^{\operatorname{lk}(K,\gamma)} \rho(\gamma)
\]
for $\gamma\in\pi_1(S^3\setminus K)$ defines a $GL_2(\C)$
representation satisfying the condition of Theorem~\ref{thm:rep} with
$\mu_0 = \mu^2$ and $\lambda_0 = \lambda$. (Here $\operatorname{lk}(K,\gamma)$ is the linking number of $K$ with $\gamma$, i.e.,
the image of $\gamma$ in $H_1(S^3\setminus K) \cong \Z$.)
Deduce Theorem~\ref{thm:Apoly}.
\item
For $K = T(3,4)$ and $\lambda_0 = \mu_0^{-8}$ with arbitrary
$\mu_0\in\C^*$, find a $GL_3(\C)$ representation of $\pi_1(S^3
\setminus K) \cong \langle x,y \, | x^3=y^4 \rangle$ satisfying the
condition of Theorem~\ref{thm:rep}. (Note that in this presentation,
$m=xy^{-1}$ and $l=x^3m^{-12}$.) This shows that $\lambda\mu^8-1$
is a factor of $\Aug_{T(3,4)}(\lambda,\mu)$; as discussed above, this
factor does not appear in the $A$-polynomial of $T(3,4)$.
\end{enumerate}
\end{exercise}

We now turn to some recent
developments linking the augmentation polynomial to physics. Our discussion is very sketchy and imprecise; see
\cite{AV,AENV} for more details. Recently the (three-variable)
augmentation polynomial has appeared in various string theory papers
\cite{AV,FGS}, in the context of studying topological strings for
$SU_N$ Chern--Simons theory on $S^3$. A very sketchy description of the
idea, whose origins in the physics literature include \cite{GV,OV}, is
as follows.

Start with
a knot $K \subset S^3$, with conormal bundle $L_K \subset
T^*S^3$. (Note that this differs slightly from our usual setting of $K \subset \R^3$, though not in a substantial way, either topologically or contact-geometrically.)
Collapse the zero section of $T^*S^3$ to a point,
resulting in a conifold singularity; we can
then resolve the singularity to a $\mathbb{CP}^1$ to obtain the ``resolved
conifold'' given as the total space of the bundle
 \[
 \mathcal{O}(-1) \oplus \mathcal{O}(-1) \rightarrow \mathbb{CP}^1.
 \]
(In physics language, this conifold transition is motivated by placing $N$ branes on the zero section of $T^*S^3$ and taking the $N\to\infty$ limit.) One would like to follow $L_K$ through this conifold
transition to obtain a special Lagrangian $\tilde{L}_K \subset
\mathcal{O}(-1) \oplus \mathcal{O}(-1)$. In \cite{AV}, Aganagic and Vafa propose a generalized SYZ conjecture by which $\tilde{L}_K$
produces a mirror Calabi--Yau of $\mathcal{O}(-1) \oplus
\mathcal{O}(-1)$ given by a variety of the form
\[
uv = \textbf{\textit{A}}_K(e^x,e^p,Q)
\]
where $(u,v,x,p) \subset \C^4$, $Q$ is a parameter measuring the
complexified K\"ahler class of $\mathbb{CP}^1$, and
$\textbf{\textit{A}}_K$ is a three-variable polynomial that Aganagic
and Vafa \cite{AV}
refers to as the ``$Q$-deformed $A$-polynomial''.\footnote{
In a related vein, Fuji, Gukov, and Su{l}kowski \cite{FGS} have
proposed a four-variable ``super-$A$-polynomial'' that specializes to
the $Q$-deformed $A$-polynomial.}

Surprisingly, we can make the following conjecture, for which there is
strong circumstantial evidence \cite{AENV}:

\begin{conjecture}[\cite{AV,AENV}]
The three-variable augmentation polynomial
\label{conj:physics}
and the $Q$-deformed
$A$-polynomial agree for all $K$:
\[
\textbf{\textit{A}}_K(e^x,e^p,Q)
= \Aug_K(\lambda=e^x,\mu=e^p,U=Q).
\]
\end{conjecture}

Although Conjecture~\ref{conj:physics} has yet to be rigorously
proven, it would have significant implications for the augmentation
polynomial. By physical arguments (see in particular \cite{GSV} and
\cite{AV}), $\textbf{\textit{A}}_K$ satisfies a number of interesting
properties. In particular, $\textbf{\textit{A}}_K$ encodes a large
amount of information about the knot $K$, possibly including the
HOMFLY-PT polynomial as well as
Khovanov--Rozansky HOMFLY-PT homology \cite{KhRHOMFLY} and
other knot homologies (or some portion thereof). The knot homologies appear
in studying Nekrasov deformation of topological strings and
refined Chern--Simons theory \cite{GSV}.

Thus, assuming Conjecture~\ref{conj:physics}, one can make purely
mathematical predictions about the augmentation polynomial. One such
prediction begins with the observation (whose proof we omit here) that
for any knot $K$,
\[
\Aug_K(\lambda=0,\mu=U,U) = 0
\]
for all $U$. It appears that the first-order behavior of the
augmentation variety near the curve $\{(0,U,U)\} \subset (\C^*)^3$
determines a certain specialization of the HOMFLY-PT polynomial:

\begin{conjecture}
Let $K$ be any knot in $S^3$. \label{conj:HOMFLY}
Let $f(U)$ be the polynomial such that
near $(\lambda,\mu,U)=(0,U,U)$, the zeroes of the augmentation
polynomial $\Aug_K$ satisfy
\[
\mu = U + f(U) \lambda + O(\lambda^2)
\]
($f(U)$ can be explicitly written in terms of the
$\lambda^1$ and $\lambda^0$ coefficients of $\Aug_K$). Then
\[
\frac{f(U)}{U-1} = P_K(U^{-1/2},1),
\]
where $P_K(a,q)$ is the HOMFLY-PT polynomial of $K$ (sometimes written
as $P_K(a,z = q-q^{-1})$).
\end{conjecture}

\noindent
Conjecture~\ref{conj:HOMFLY} has been checked for all knots where the
augmentation polynomial is currently known, including many where the
$Q$-deformed $A$-polynomial has not been computed.

\begin{exercise}
Verify Conjecture~\ref{conj:HOMFLY} for the unknot and the
right-handed and left-handed trefoils, using the augmentation polynomials computed in
Exercise~\ref{exc:augpolycomputations}.
\label{exc:HOMFLY}
Note that
the HOMFLY-PT polynomials for the unknot and the RH trefoil are $1$
and 
$-a^{-4}+a^{-2}q^{-2}+a^{-2}q^2$, respectively.
\end{exercise}

In a different direction, the physics discussion of
$\textbf{\textit{A}}_K$ in \cite{AV} also predicts that the
augmentation polynomial is determined by the recurrence relation for
the colored HOMFLY-PT polynomials:

\begin{conjecture}
Let $\{P_{K;n}(a,q)\}$ denote the colored HOMFLY-PT polynomials of $K$, colored by the $n$-th symmetric power of the fundamental representation.
\label{conj:qhol}
Define
operations $\widehat{\lambda}$, $\widehat{\mu}$ by
$\widehat{\lambda}(P_{K;n}(a,q)) = P_{K;n+1}(a,q)$ and
$\widehat{\mu}(P_{K;n}(a,q)) = q^n P_{K;n}(a,q)$. These polynomials
satisfy a minimal recurrence relation of the form
\[
\widehat{A}_K(a,q,M,L) P_{K;n}(a,q) = 0,
\]
where $\widehat{A}_K$ is a polynomial in noncommuting variables
$L$, $M$ and commuting parameters $a,q$;
see \cite{Garqholonomic}. Then sending $q \to 1$ and applying an appropriate change of variables sends $\widehat{A}_K(a,q,M,L)$ to the augmentation polynomial $\Aug_K(\lambda,\mu,U)$.
\end{conjecture}

\noindent
The precise change of variables depends on the conventions used for $P_{K;n}(a,q)$. In the conventions of \cite{FGS} (where their $x,y$ are our $M,L$), a more exact statement is that $\Aug_K(\lambda,\mu,U)$ and
\[
\widehat{A}_K\left(a=U,\,q=1,\,M=\mu^{-1},\,L=\frac{\mu-1}{\mu-U}\lambda\right)
\]
agree up to trivial factors.

Conjecture~\ref{conj:qhol} is a direct analogue of the AJ conjecture
\cite{GaroufalidisAJ} (quantum volume conjecture, in the physics
literature) relating
colored Jones polynomials to the $A$-polynomial, with colored
HOMFLY-PT replacing colored Jones, and the augmentation polynomial
replacing the $A$-polynomial. See also \cite{FGS} for an extended discussion of this topic.

\section{Transverse Homology}
\label{sec:transhom}

In this section, we discuss a concrete application of knot contact
homology to contact topology, and in particular to transverse
knots. Here one obtains additional filtrations on the knot DGA that
produce effective invariants of transverse knots. So far our construction of knot contact homology begins with a smooth knot in $\R^3$; we now explore what happens if the knot is assumed to be transverse to a contact structure on $\R^3$ (note that this is independent of the canonical contact structure on $ST^*\R^3$!).

\begin{definition}
Let $\xi = \ker(\alpha=dz+r^2d\theta)$ be the standard contact
structure on $\R^3$. An oriented knot $T \subset \R^3$ is
\textit{transverse} if $\alpha>0$ along $T$.
\end{definition}

One usually studies transverse knots up to \textit{transverse
  isotopy}: isotopy through transverse knots.
There is a standard transverse unknot in $\R^3$ given by the unit
circle in the $xy$ plane. By work of Bennequin \cite{Bennequin}, any braid
produces a transverse knot by gluing the closure of the braid into a
neighborhood of the standard unknot. Conversely, all transverse knots are
obtained in this way, up to transverse isotopy: the map from braids to
transverse knots is surjective. The following theorem precisely
characterizes failure of injectivity.

\begin{theorem}[Transverse Markov Theorem \cite{OSh,Wrinkle}]
Two braids produce transverse knots that are transversely isotopic if
and only if they are related by:
\begin{itemize}
\item
conjugation in the braid groups
\item
positive Markov stabilization and destabilization: $(B \in B_n)
\longleftrightarrow (B\sigma_n \in B_{n+1})$.
\end{itemize}
\end{theorem}

Transverse knots have two ``classical'' invariants of transverse
knots:
\begin{itemize}
\item
underlying topological knot type
\item self-linking number (for a braid, $sl=w-n$).
\end{itemize}
It is of considerable interest to find other, ``effective''
transverse invariants, which can distinguish between transverse knots
with the same classical invariants. One such invariant is the
transverse invariant in knot Floer homology \cite{OST,LOSS}. This (more precisely,
one version of it) associates, to a transverse knot $T$ of
topological type $K$, an element $\widehat{\theta}(T) \in
\widehat{\textit{HFK}}(m(K))$. The HFK invariant has been shown to be
effective at distinguishing transverse knots; see e.g. \cite{NOT}.

The purpose of this section is to discuss how one can refine knot
contact homology to produce another effective transverse
invariant. Geometrically, the idea is as follows (see
\cite{EENStransverse} for details).
Given a transverse knot $T\subset(\R^3,\xi)$, one constructs the conormal
bundle $\Lambda_T \subset ST^*R^3$ as usual. Now the cooriented
contact plane field $\xi$ on $\R^3$ also has a conormal lift
$\widetilde{\xi} \subset ST^*\R^3$: concretely, this is the section of
$ST^*\R^3$ given by $\alpha/|\alpha|$ where $\alpha$ is the contact
form. Since $T$ is transverse to $\xi$, $\Lambda_T \cap
\widetilde{\xi} = \emptyset$.

One can choose an almost complex structure on the symplectization
$\R\times ST^*\R^3$ (and change the metric on $\R^3$ that determines
$ST^*\R^3$) so that $\R\times\widetilde{\xi}$ is
holomorphic. Given a holomorphic disk with boundary on
$\R\times\Lambda_T$ as in the LCH of $\Lambda_T$, one can then count
intersections with $\R\times\widetilde{\xi}$, and all of these
intersections are positive. Thus we can filter the LCH differential of
$\Lambda_T$:
\[
\d(a_i) = \sum_{\dim \M(a_i;a_{j_1},\ldots,a_{j_k})/\R = 0} ~~\sum_{\Delta\in\M/\R} (\textrm{sgn}) U^{\#(\Delta\cap(\R\times\widetilde{\xi}))}
e^{[\d\Delta]} a_{j_1}\cdots a_{j_k}.
\]
Here $[\d \Delta]$ is the homology class of $\d \Delta$ in
$H_1(\Lambda_T)$ and $\#(\Delta\cap(\R\times\widetilde{\xi}))$ is
always nonnegative.
This gives a filtered version for the knot DGA for $T$, which is now a
DGA over $R_0[U]$ (recall that $R_0 = \Z[\lambda^{\pm 1},\mu^{\pm 1}]$).

\begin{definition}
The \textit{transverse DGA} $(\A^-,\d^-)$ associated to a transverse
knot $T \subset \R^3$ is the resulting DGA over $R_0[U]$.
\end{definition}

\noindent
The minus signs in the notation $(\A^-,\d^-)$ are by analogy with Heegaard Floer
homology.

When the transverse knot $T$ is the closure of a braid $B$, there is a
straightforward combinatorial description for the transverse
DGA:

\begin{definition}
Let $B$ be a braid. The \textit{combinatorial transverse DGA}
for $B$ is the DGA over $R_0[U]$ with the same
generators and differential as in Definition~\ref{def:knotDGA}, but
with
$\LL=\operatorname{diag}(\lambda\mu^w,1,\ldots,1)$ rather than
$\operatorname{diag}(\lambda\mu^{w}U^{-(w-n+1)/2},1,\ldots,1)$.
\end{definition}

\noindent
With this new definition of $\LL$, the differential in
Definition~\ref{def:knotDGA} contains only nonnegative powers of $U$,
and we indeed obtain a DGA over $R_0[U]$ (versus $R_0[U^{\pm 1}]$ in
Definition~\ref{def:knotDGA}).

\begin{theorem}[\cite{EENStransverse}]
The transverse DGA and the combinatorial transverse DGA agree.
\end{theorem}

We now have the following invariance result.

\begin{theorem}[\cite{EENStransverse,transhom}]
Given a braid $B$, the DGA $(\A^-,\d^-)$ over $R_0[U]$, up to stable
tame isomorphism, is an invariant of the transverse knot corresponding to $B$.
\label{thm:transhom}
\end{theorem}

Theorem~\ref{thm:transhom} follows from the general theory of
Legendrian contact homology (and a few
details that we omit here). Alternatively, one can prove directly
that the combinatorial transverse DGA is a transverse invariant by
checking invariance under braid conjugation and positive braid
stabilization, and invoking the Transverse Markov Theorem; this
approach is carried out in \cite{transhom}. In any case, the homology
of $(\A^-,\d^-)$ is also a transverse invariant and is called
\textit{transverse homology}.

\begin{remark}
In fact, a transverse knot gives \textit{two} filtrations on the knot
DGA, given by $U$ and another parameter $V$; what we have presented is
the specialization $V=1$. One can extend this to a DGA over
$R_0[U,V]$ that, like $(\A^-,\d^-)$, has a combinatorial
description. The generators of the DGA are the usual ones from
Definition~\ref{def:knotDGA}, while the differential is given by:
\begin{align*}
\d(\AA) &= 0 \\
\d(\BB) &= \AA - \LL \cdot \phi_B(\AA) \cdot \LL^{-1} \\
\d(\CC) &= \Ahat - \LL \cdot \Phil_B \cdot \Acheck \\
\d(\DD) &= \Acheck - \Ahat \cdot \Phir_B \cdot \LL^{-1} \\
\d(\EE) &= \hat{\BB} - \CC - \LL \cdot \Phil_B \cdot \DD \\
\d(\FF) &= \check{\BB} - \DD - \CC \cdot \Phir_B \cdot \LL^{-1}.
\end{align*}
Here $\LL=\operatorname{diag}(\lambda\mu^{w},1,\ldots,1)$; $\AA,\Ahat,\BB,\Bhat,\CC,\DD,\EE,\FF$ are as in Definition~\ref{def:knotDGA}; and $\Acheck,\Bcheck$ are defined by:
\begin{align*}
(\Acheck)_{ij} &= \begin{cases} a_{ij} & i<j \\
-\mu V a_{ij} & i>j \\
1-\mu V & i=j \\
\end{cases}
&
(\Bcheck)_{ij} &= \begin{cases} b_{ij} & i<j \\
-\mu V b_{ij} & i>j \\
0 & i=j. \\
\end{cases}
\end{align*}
Geometrically, the powers of $V$ count intersections with the
``negative'' lift of $\xi$
to $ST^*\R^3$, given by $-\alpha/|\alpha|$.
The full DGA over $R_0[U,V]$ has some nice formal properties, such as
its behavior under transverse stabilization, but for known applications it
suffices to set $V=1$ and thus ignore $V$.
\end{remark}

We now return to the transverse DGA $(\A^-,\d^-)$ over $R_0[U]$.
In a manner familiar from Heegaard Floer theory, one can obtain
several other flavors of transverse homology from $(\A^-,\d^-)$. Two
particularly interesting ones are:
\begin{itemize}
\item
The ``hat version'': $(\widehat{\A},\widehat{\d})$, a DGA over $R_0 =
\Z[\lambda^{\pm 1},\mu^{\pm 1}]$, by setting $U=0$. This is a
transverse invariant.
\item
The ``infinity version'': $(\A,\d)$, the usual knot DGA over $R =
R_0[U^{\pm 1}]$, by tensoring $(\A^-,\d^-)$ with $R_0[U^{\pm 1}]$ and replacing
$\lambda$ by $\lambda U^{-(w-n+1)/2}$. This is an invariant of the
underlying topological knot, as usual.
\end{itemize}

\begin{remark}
Independent of the fact that the infinity version is the usual knot
DGA, we can see geometrically that the infinity version is a topological
knot invariant, as follows.
If we disregard positivity of intersection, then powers of $U$ in the
differential $\d$ merely encode homological data about the holomorphic
disk $\Delta$; a bit of thought shows that
$\#(\Delta\cap(\R\times\widetilde{\xi}))$ is equal to the class of
$\Delta$ in $H_2(S^2) \cong \Z$. Thus this indeed reduces to the usual
LCH DGA of $\Lambda_K$.
\end{remark}

We now have the following result.

\begin{theorem}[\cite{EENStransverse,transhom}]
The hat version of the transverse DGA, $(\widehat{\A},\widehat{\d})$,
\label{thm:effective}
is an effective invariant of transverse knots.
\end{theorem}

\begin{figure}
\centerline{
\includegraphics[width=5.5in]{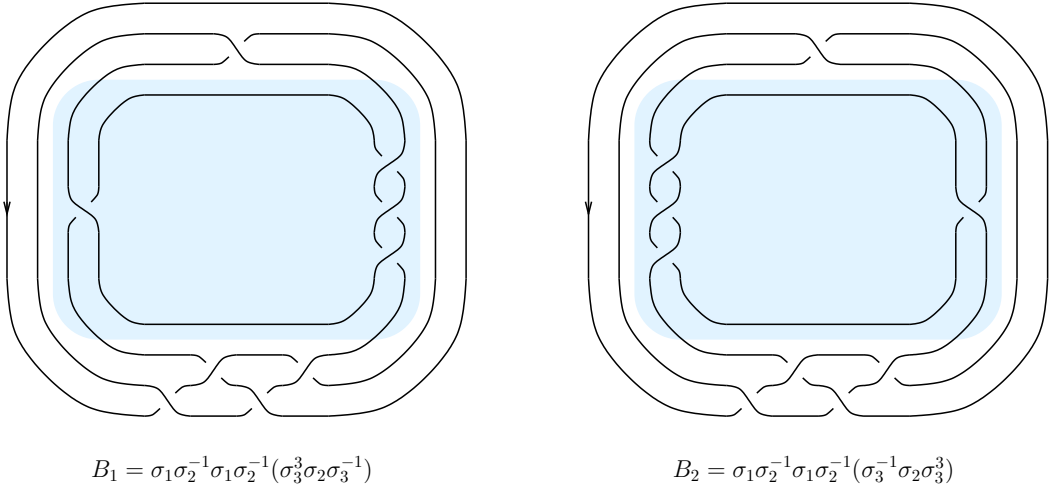}
}
\caption{
Two braids $B_1,B_2$ whose closure is the knot $m(7_6)$. To see that
they produce the same knot, note that their closures are related by a
negative flype (the shaded regions).
}
\label{fig:76}
\end{figure}

As one
example, consider the transverse knots given by the closures
of the braids $B_1,B_2$ given in Figure~\ref{fig:76}, both of which
are of topological type $m(7_6)$ and have self-linking number
$-1$. For each braid, one can count the number of augmentations of
$(\widehat{\A},\widehat{\d})$ to $\Z_3$; this augmentation number is a
transverse invariant. A
computer calculation shows that the augmentation number is $0$ for
$B_1$ and $5$ for $B_2$. It follows that the transverse knots
corresponding to $B_1$ and $B_2$ are not transversely isotopic.

One can heuristically gauge the relative effectiveness of various
transverse invariants by using the Legendrian knot atlas \cite{Atlas},
which provides a conjecturally complete list of all Legendrian knots
representing topological knots of arc index $\leq 9$.
The atlas proposes $13$ knots with arc index $\leq 9$ that have at least two transverse representatives with the same self-linking number. Of these $13$:
\begin{itemize}
\item
$6$ ($m(7_2)$, $m(10_{132})$, $m(10_{140})$, $m(10_{145})$, $m(10_{161})$, $12n_{591}$) have transverse representatives that can be distinguished by both the HFK invariant and by transverse homology;
\item
$4$ ($m(7_6)$, $9_{44}$, $9_{48}$, $10_{136}$) can be distinguished by transverse homology but not the HFK invariant;
\item
$3$ ($m(9_{45})$, $10_{128}$, $10_{160}$) cannot yet be distinguished by either HFK or transverse homology.
\end{itemize}
Of these last $3$, preliminary joint work with Dylan Thurston suggests
that $m(9_{45})$ and $10_{128}$ can be distinguished by naturality in conjunction with the
HFK invariant, but the third cannot.\footnote{The transverse representatives of $m(7_6)$, $9_{44}$, $9_{48}$, $10_{136}$, and $10_{160}$ cannot be distinguished by the transverse HFK invariant, with or without
  naturality, because $\widehat{\textit{HFK}}=0$ and $\textit{HFK}^-$ has rank $1$ in the relevant bidegree.
}
It is conceivable that some or all of these last $3$ can be
distinguished by transverse homology, but they are related by an
operation known as ``transverse mirroring'' that is relatively difficult to
detect by transverse homology.

It appears that the two known effective transverse invariants, the
transverse HFK invariant and transverse homology, are functionally
independent, but it would be very interesting to know if there is some
connection between them.

\section*{Appendix: Conventions and the Fully Noncommutative DGA}

In the literature on knot contact homology, a number of mutually inconsistent conventions are used. The conventions that we have adopted in this article are unfortunately different again from the existing ones, but we would like to advocate these new conventions as combining the best qualities of previous ones while avoiding some disadvantages that have become apparent in the interim.

First we describe how to extend the definition of knot contact homology from Section~\ref{sec:kch} in two directions: first, by allowing for multi-component links, and second, by extending to the fully noncommutative DGA (see Remark~\ref{rmk:noncommDGA}), in which homology classes do not commute with Reeb chords. The result is a ``stronger'' formulation of (combinatorial) knot contact homology than usually appears in the literature. After this, we will discuss how this definition compares to previous conventions.

If $K$ is a link given by the closure of a braid $B \in B_n$, we can define a slightly more complicated version of the braid homomorphism $\phi_B$ from Section~\ref{sec:kch} as follows. Let $\tilde{\A}_n$ denote the tensor algebra over $\Z$ freely generated by $a_{ij}$, $1\leq i\neq j\leq n$, and by $\tilde{\mu}_i^{\pm 1}$, $1\leq i\leq n$. (Here the $\tilde{\mu}_i$'s do not commute with the $a_{ij}$'s, or indeed with each other, and the only nontrivial relations are $\tilde{\mu}_i \cdot \tilde{\mu}_i^{-1} = \tilde{\mu}_i^{-1} \cdot \tilde{\mu}_i = 1$.) For $1\leq k\leq n-1$, define $\phi_{\sigma_k} :\thinspace \tilde{\A}_n \to \tilde{\A}_n$ by:
\[
\phi_{\sigma_k} :\thinspace
\begin{cases}
a_{ij} \mapsto a_{ij}, & i,j \neq k,k+1 \\
a_{k+1,i} \mapsto a_{ki}, & i\neq k,k+1 \\
a_{i,k+1} \mapsto a_{ik}, & i\neq k,k+1 \\
a_{k,k+1} \mapsto -a_{k+1,k} & \\
a_{k+1,k} \mapsto -\tilde{\mu}_k a_{k,k+1}\tilde{\mu}_{k+1}^{-1} & \\
a_{ki} \mapsto a_{k+1,i}-a_{k+1,k}a_{ki}, & i\neq k,k+1 \\
a_{ik} \mapsto a_{i,k+1}-a_{ik}a_{k,k+1}, & i<k \\
a_{ik} \mapsto a_{i,k+1}-a_{ik} \tilde{\mu}_k a_{k,k+1} \tilde{\mu}_{k+1}^{-1}, & i>k+1 \\
\tilde{\mu}_i^{\pm 1} \mapsto \tilde{\mu}_i^{\pm 1}, & i\neq k,k+1 \\
\tilde{\mu}_k^{\pm 1} \mapsto \tilde{\mu}_{k+1}^{\pm 1} & \\
\tilde{\mu}_{k+1}^{\pm 1} \mapsto \tilde{\mu}_k^{\pm 1}. &
\end{cases}
\]
This extends to a group homomorphism $\phi :\thinspace B_n \to \Aut\tilde{\A}_n$ and thus defines a map $\phi_B \in \Aut\tilde{\A}_n$.

Suppose that $K$ has $r$ components, and number the components of $K$ $1,\ldots,r$. For $i=1,\ldots,n$, define $\alpha(i) \in \{1,\ldots,r\}$ to be the number of the component containing strand $i$ of the braid $B$ whose closure is $K$.
If we now define $\A_n$ to be the tensor algebra over $\Z$ freely generated by the $a_{ij}$'s and by variables $\mu_1^{\pm 1},\ldots,\mu_r^{\pm 1}$, then it is easy to check that $\phi_B$ descends to an algebra automorphism of $\A_n$ by setting $\tilde{\mu}_i = \mu_{\alpha(i)}$ for all $1\leq i\leq n$. We can define $\Phi^L_B,\Phi^R_B \in \operatorname{Mat}_{n\times n}(\A_n)$ as in Definition~\ref{def:PhiLPhiR}, with the important caveat that the extra strand $\ast$ is treated as strand $0$ rather than strand $n+1$; for multi-component links, this makes a difference because of the form of the definition of $\phi_{\sigma_k}$ above.

Define $\A$ to be the tensor algebra over $\Z[U^{\pm 1}]$ freely generated by $\mu_1^{\pm 1},\ldots,\mu_r^{\pm 1}$ along with the generators $a_{ij},b_{ij},c_{ij},d_{ij},e_{ij},f_{ij}$ as in Definition~\ref{def:knotDGA}. Assemble $n\times n$ matrices $\AA,\hat{\AA},\BB,\hat{\BB},\CC,\DD,\EE,\FF$, where $\CC,\DD,\EE,\FF$ are as in Definition~\ref{def:knotDGA}, while
\begin{align*}
\AA_{ij} &= \begin{cases} a_{ij} & i<j \\
-a_{ij} \mu_{\alpha(j)} & i>j \\
1-\mu_{\alpha(i)} & i=j \\
\end{cases}
&
\BB_{ij} &= \begin{cases} b_{ij} & i<j \\
-b_{ij} \mu_{\alpha(j)} & i>j \\
0 & i=j \\
\end{cases}  \\
(\Ahat)_{ij} &= \begin{cases} U a_{ij} & i<j \\
-a_{ij} \mu_{\alpha(j)} & i>j \\
U-\mu_{\alpha(i)} & i=j \\
\end{cases}
&
(\Bhat)_{ij} &= \begin{cases} U b_{ij} & i<j \\
-b_{ij} \mu_{\alpha(j)} & i>j \\
0 & i=j. \\
\end{cases}
\end{align*}
Also define a matrix $\LL$ as follows:
choose one strand of $B$ belonging to each component of the closure
$K$, and call the resulting $r$ strands \textit{leading}; then define
\[
(\LL)_{ij} = \begin{cases} \lambda_{\alpha(i)} \mu_{\alpha(i)}^{w(\alpha(i))}
U^{-(w(\alpha(i))-n(\alpha(i))+1)/2} & i=j \text{ and strand $i$ leading} \\
1 & i=j \text{ and strand $i$ not leading} \\
0 & i \neq j,
\end{cases}
\]
where $n(\alpha)$ is the number of strands belonging to component
$\alpha$ and $w(\alpha)$ is the writhe of component $\alpha$ viewed as
an $n(\alpha)$-strand braid (with the other components deleted).

With this notation, one can now define the differential $\d$ on $\A$
exactly as in Definition~\ref{def:knotDGA}. The resulting DGA has the
same properties as in Theorem~\ref{thm:invariance}: $\d^2=0$ and
$(\A,\d)$ is an isotopy invariant of the link
$K$ viewed as an oriented link with numbered components, up to stable
tame isomorphisms that act as the identity on $U$ and on each of
$\lambda_1,\ldots,\lambda_r,\mu_1,\ldots,\mu_r$.

Note that the definition of the DGA given above is for the topological
knot/link invariant as discussed in Sections~\ref{sec:LCH}
and~\ref{sec:kch}. This corresponds to ``infinity transverse
homology'' from \cite{transhom} (also mentioned in
\cite{EENStransverse}). There is an analogous definition of transverse
homology as in Section~\ref{sec:transhom} or
\cite{EENStransverse,transhom} but we omit its definition here.

We now compare our definition to the two previous conventions for the
knot DGA: the convention from \cite{EENStransverse,EENS} and the
convention from \cite{transhom,Ngframed}. Note that all versions of the DGA
from these references first quotient so that
homology classes $\lambda_1,\ldots,\lambda_r,\mu_1,\ldots,\mu_r$
commute with all Reeb chords. Also, the versions from \cite{EENStransverse} and
\cite{transhom} involve an additional variable $V$, but we set $V=1$
for this discussion; as explained in \cite[\S 4]{transhom}, this does
not lose any information. Finally, the conventions from \cite{EENS}
and \cite{Ngframed} agree with the conventions from
\cite{EENStransverse} and \cite{transhom}, respectively, after setting
$U=V=1$.

We claim that we can then obtain the knot DGAs in the conventions of \cite{EENStransverse} and \cite{transhom} from the knot DGA presented in this article, up to isomorphism, as follows:
\begin{itemize}
\item
for \cite{EENStransverse}, replace $\lambda_\alpha \mapsto -\lambda_\alpha$ and $\mu_\alpha \mapsto -\mu_\alpha$ for each component $\alpha$;
\item
for \cite{transhom}, which only considers the single-component case, keep $\lambda$ as is, and replace $\mu \mapsto -\mu^{-1}$.
\end{itemize}

We check the claim for the convention of \cite{EENStransverse}; the claim for \cite{transhom} then follows from \cite[\S 3.4]{transhom}.\footnote{Note that
\cite{transhom}, building on work from \cite{Ngframed}, uses an
unusual convention for braids, so that a \textit{positive} generator
$\sigma_k$ of the braid group is given topologically as a
\textit{negative} crossing in the usual knot theory sense. This has
the effect of mirroring all topological knots and explains the
$\mu^{-1}$ difference in conventions.} Note that negating each $\lambda_\alpha$ and $\mu_\alpha$ causes our definitions to line up
precisely with the definitions from \cite{EENStransverse}, except for $\LL$ (denoted in \cite{EENStransverse} by -$\boldsymbol{\lambda}$), which differs by the presence or absence of a power of $U$, along with
some signs. But the power of $U$ is merely a notational/framing issue (cf.\ \cite{transhom}), while the sign discrepancy
disappears because it can be
checked that the products of the diagonal entries of $\LL$ and $-\boldsymbol{\lambda}$ corresponding to any particular link component are exactly equal including sign,
whence the DGAs given by the two conventions are isomorphic by the argument of
\cite[Proposition~3.1]{transhom}.

\begin{remark}
Except for the $\mu \mapsto \mu^{-1}$ issue, all differences
between conventions consist just of negating some subset of
$\{\lambda,\mu\}$. This is explained by the fact that the signs in the
differential in Legendrian contact homology depend on a choice of spin
structure on the Legendrian submanifold $\Lambda$; see \cite{EES05b}
for full discussion. In our setting, if $K$ is a knot, $\Lambda_K
\cong T^2$ has four spin structures, and changing from one spin
structure to another sends $\lambda \mapsto \pm \lambda$ and $\mu
\mapsto \pm \mu$. Thus the different choices of signs arise from
different choices of spin structure on $\Lambda_K$.
\end{remark}

A summary of the relations between conventions in different articles is as follows:
\[
\xymatrix{
*\txt{This article} \ar[rr]^{\lambda\mapsto-\lambda,~\mu\mapsto-\mu} \ar[drr]_-{\mu\mapsto-\mu^{-1}} && *\txt{\cite{EENStransverse}} \ar[r]^{U\mapsto 1} & *\txt{\cite{EENS}} &&\\
&& *\txt{\cite{transhom}} \ar[r]^{U \mapsto 1} & *\txt{\cite{Ngframed}} \ar[rr]^-{\lambda\mapsto 1,~\mu\mapsto 1} && *\txt{\cite{NgKCH1,NgKCH2}.}
}
\]

We close by noting that our current choice of conventions allows for
some cleaner results than the conventions from
\cite{EENStransverse} or \cite{transhom}.
In particular, our signs are more natural than the signs from either \cite{EENStransverse} or \cite{transhom} when we consider the relation to representations of the knot group as in Section~\ref{sec:aug}. For instance, our two-variable augmentation polynomials are divisible by $(\lambda-1)(\mu-1)$ as opposed to $(\lambda+1)(\mu+1)$ in
\cite{EENStransverse} or $(\lambda-1)(\mu+1)$ in \cite{transhom,Ngframed}, and $A_K(\lambda,\mu)$ divides $\Aug_K(\lambda,\mu^2)$ in our convention rather than
$\Aug_K(-\lambda,-\mu^2)$ or $\Aug_K(\lambda,-\mu^2)$ in the other two.

There is another technical reason for preferring our signs or those from \cite{EENStransverse} to the ones from \cite{transhom}, or more precisely to the extrapolation of \cite{transhom} to the link case. In either of the first two cases but not the third, we have the following statement, which we leave as an exercise.

\begin{proposition}
Let $K$ be a link given by the closure of braid $B$, and let $K'
\subset K$ be a sublink given by the closure of a subbraid $B' \subset B$
obtained by erasing some strands of $B$. Then the DGA for $K'$ is a
quotient of the DGA for $K$, given by setting all Reeb chords
$a_{ij},b_{ij},$ etc. to $0$ unless strands $i$ and $j$ both belong to $B'$.
\end{proposition}

\noindent
This result is used in \cite{AENV} and is a special case of a general result relating the
Legendrian contact homology of a multi-component Legendrian to the LCH
of some subset of components.


\bibliographystyle{alpha}
\bibliography{biblio}

\end{document}